\documentclass[11pt]{amsart}

\usepackage{amsmath,amssymb}
\usepackage[all]{xy}
\usepackage{pgf}
\usepackage{tikz}

\def\eps{\varepsilon}

\def\be{\begin{equation}}
\def\ee{\end{equation}}
\def\ba{\begin{align}}
\def\bm{\begin{multline}}
\def\bfig{\begin{figure}[htb]}
\def\efig{\end{figure}}
\numberwithin{equation}{section}
\newtheorem{theorem}{Theorem}[section]
\newtheorem{proposition}[theorem]{Proposition}
\newtheorem{lemma}[theorem]{Lemma}
\newtheorem{corollary}[theorem]{Corollary}

\DeclareMathSymbol{\leqslant}{\mathalpha}{AMSa}{"36}
\DeclareMathSymbol{\geqslant}{\mathalpha}{AMSa}{"3E}
\DeclareMathSymbol{\doteqdot}{\mathalpha}{AMSa}{"2B}
\DeclareMathSymbol{\circlearrowright}{\mathalpha}{AMSa}{"08}
\DeclareMathSymbol{\subsetneq}{\mathalpha}{AMSb}{"28}
\DeclareMathSymbol{\supsetneq}{\mathalpha}{AMSb}{"29}
\renewcommand{\leq}{\;\leqslant\;}
\renewcommand{\geq}{\;\geqslant\;}

\newcommand{\upchi}{\raise 2pt \hbox{$\chi$}}
\def\writefig#1 #2 #3 {\rlap{\kern #1 truecm \raise #2 truecm
\hbox{#3}}}

\newcommand{\caC}{{\mathcal C}}

\newcommand{\caE}{{\mathcal E}}
\newcommand{\caF}{{\mathcal F}}

\newcommand{\caO}{{\mathcal O}}
\newcommand{\caP}{{\mathcal P}}

\newcommand{\bbE}{{\mathbb E}}

\newcommand{\bbN}{{\mathbb N}}

\newcommand{\bbP}{{\mathbb P}}

\newcommand{\bbR}{{\mathbb R}}

\newtheorem{definition}{Definition}
\newcommand{\p}{\mathbb{P}}
\usepackage{anysize}
\marginsize{3cm}{3cm}{3cm}{3cm}

\begin{document}

\title{Multi-scale metastable dynamics and the asymptotic stationary distribution 
of perturbed Markov chains}

\author{Volker Betz and St\'ephane Le Roux}

\address{Volker Betz \hfill\newline
\indent Fachbereich Mathematik \hfill\newline
\indent TU Darmstadt \hfill\newline
\indent Schlossgartenstrasse 7, 64289 Darmstadt \hfill\newline
{\small\rm\indent http://www.mathematik.tu-darmstadt.de/$\sim$betz/}
}
\hfill\newline
\email{betz@mathematik.tu-darmstadt.de}

\address{St\'ephane Le Roux \hfill\newline
\indent D\'epartement d'Informatique \hfill\newline
\indent Universit\'e Libre Du Bruxelles \hfill\newline
\indent ULB CP212, Boulevard du Triomphe, 1050 Bruxelles \hfill\newline
{\small\rm\indent http://www.ulb.ac.be/di/verif/sleroux/}
}
\email{stephane.le.roux@ulb.ac.be}

\begin{abstract}
We consider a simple but important class of metastable discrete time Markov chains,
which we call perturbed Markov chains. Basically, we assume that the transition
matrices depend on a parameter $\eps$, and converge as $\eps \to 0$.  
We further assume that the chain is irreducible for $\eps > 0$ but may have 
several essential communicating classes when $\eps=0$. This leads to metastable 
behavior, possibly on multiple time scales. 
For each of the relevant time scales, we derive two effective chains. 
The first one describes the (possibly irreversible) metastable dynamics, while the second
one is reversible and describes metastable escape probabilities. 
Closed probabilistic expressions are given for the asymptotic 
transition probabilities of these chains, 
but we also show how to compute them in a fast and numerically stable way. As a consequence,
we obtain efficient algorithms for computing the committor function and the limiting
stationary distribution. 

\end{abstract}

\maketitle

\noindent
{\footnotesize {\it Keywords:} escape times, 
non-reversible Markov chains, asymptotics}

\noindent
{\footnotesize {\it 2000 Math.\ Subj.\ Class.:} 60J10, 60J22}

\section{Introduction}

In this paper we give a detailed analysis of the asymptotic dynamics and 
stationary distribution for a special class of metastable Markov chains. 
Loosely speaking, a metastable Markov chain 
is one that, on short time scales, looks like a stationary Markov chain exploring 
only a small subset of its state space; on longer time scales, however, it performs fast and 
rare transitions between different such subsets. 
 
The topic of metastability is an old one.
Its origins can be traced back at least to the works of Eyring \cite{Eyr35} and 
Kramers \cite{Kra40}, who studied it in the context of chemical reaction rates. In 
the context of perturbed dynamical systems, Freidlin and Wentzell \cite{FW98} developed
a systematic approach based on large deviation theory. This approach was extended by
Berglund and Gentz \cite{BGBook} to cover stochastic bifurcation and stochastic 
resonance, and by Olivieri and Scoppola \cite{OlSc1, OlSc2} 
to study dynamics of Markov chains with exponentially small transition probabilities. 
Bovier, Eckhoff,
Gayrard and Klein \cite{BEGK1, BEGK2, BGK} developed a systematic approach 
based on capacities, and gave a precise mathematical definition for
metastability. The transition path theory  
\cite{vdE1,vdE2} investigates the most probable paths that the Markov chain uses when 
travelling between different metastable states. Recent books on
various aspects of metastability include the monograph \cite{OlVa12}, 
and the lecture notes \cite{BovNotes}.

As we will discuss in Section \ref{dynamics}, the 
chains treated in the present paper are metastable in the sense 
of Bovier et al.  
Our situation is considerably simpler than the general one: 
the state space is of fixed finite (but possibly large) size, 
and the metastability enters via an explicit parameter in the transition matrix. 
In contrast, the theory described in \cite{BEGK1, BEGK2, BGK} is built to 
accommodate the difficult situation where metastability is not necessarily a consequence
of some transition probabilities becoming small, but may also arise from a limit where the 
number of states diverges. In the case of reversible Markov chains, many of our main 
results can be deduced from the theory of \cite{BEGK1, BEGK2, BGK}, although our proofs are  
different and do not rely on the variational methods used there. 
The benefit of this is that our methods also cover the non-reversible situation where 
Dirichlet-form techniques are less useful. 

Let us describe our setup and results in some more detail. 
Consider a family of discrete time Markov chains 
$X^{(\eps)} = (X^{(\eps)}_n)_{n \in \bbN_0}$ with finite state space $S$ and 
transition matrices $P_\eps = (p_\eps(x,y))_{x,y \in S}$. 
We assume that the map $\eps \mapsto p_\eps(x,y)$ is 
continuous at $\eps = 0$ for all $x,y \in S$, and that the Markov 
chain $X^{(\eps)}$ is irreducible when 
$\eps > 0$. For $\eps=0$ however, the chain may have several essential 
communicating classes. Such 
a family of Markov chains is called an irreducible perturbation of $X^{(0)}$, 
or simply an irreducibly perturbed Markov chain. 

The first main result of the paper is a description of the multi-scale metastable behavior 
of the chain. Let $E_1, \ldots, E_n$ be the essential classes of the chain at parameter
$\eps = 0$. We pick $x_i \in E_i$ for all $i \leq n$ and define an effective chain 
$\hat X^{(\eps)}$ with state space $\{x_1, \ldots, x_n\}$. We prove that this chain 
captures the effective dynamics of the original chain on the shortest metastable 
time scale, in the sense that its escape probabilities and stationary distribution
are asymptotically independent of the choice of the representatives $x_1, \ldots, x_n$, 
and asymptotically equal to those of the original chain. For the stationary distribution,
this means that $\lim_{\eps \to 0} \hat \mu_\eps(x_i) / \mu_\eps(E_i) = 1$, where 
$\hat \mu_\eps$ and $\mu_\eps$ are the stationary distributions of the respective chains. 
A central tool is a natural, {\em reversible} chain that has the same 
stationary distribution as $\hat X^{(\eps)}$ and is interesting in its own right.

In order to explore longer metastable time scales, we renormalize the effective chain: 
for $\hat X^{(\eps)}$, all transitions between different states will vanish in the limit 
$\eps \to 0$. By rescaling time under suitable conditions, we obtain a new perturbed 
Markov chain, where at least one transition probability 
between distinct states is of order one as 
$\eps \to 0$. We can now iterate the procedure described above, 
yielding effective chains on smaller and smaller state spaces and encoding the dynamics of the original chain on longer and longer metastable time scales. 

A similar program has been carried out before by Olivieri and Scoppola \cite{OlSc1, OlSc2}. 
The difference to our approach is that \cite{OlSc1, OlSc2} relies on (and extends) the theory 
of Freidlin and Wentzell, while our approach is closer to the potential theoretic 
methods of Bovier et. al. \cite{BovNotes}. This allows us to avoid many of the technical
complications found in \cite{OlSc1,OlSc2}. 
Also, Olivieri and Scoppola only consider Markov chains with exponentially small transition 
probabilities, and study asymptotics on a logarithmic scale. In contrast, our methods allow
for much more general families of transition matrices, and our results are asymptotically 
sharp in the sense that we identify the correct prefactors for all our asymptotic identities. 
The last fact is particularly useful in practice, since it allows us to devise  
numerically stable algorithms for computing the asymptotic stationary distribution 
$\lim_{\eps \to 0} \mu_\eps(x)$ for all states $x \in S$. Alternatively, we can compute
the ratio of the stationary distributions for two given states $x,y$ 
without computing the full stationary distribution, thus potentially decreasing the 
computational cost considerably. These algorithms are the second main result of our work. 

To see why numerically computing the asymptotic stationary distribution 
might be a problem, consider 
the following simple example. 
Let $S = \{x,y\}$, and $P_\eps$ with elements $p_\eps(x,y) = \eps^\alpha$, 
$p_\eps(y,x) = \eps^\beta$, for some $\alpha,\beta,\eps > 0$. For $\eps = 0$, 
$\{x\}$ and $\{y\}$ are the essential classes of the chain, so 
both $x$ and $y$ are metastable. The stationary distribution of the chain is 
$\mu_\eps(x) = 
\frac{\eps^\beta}{\eps^\alpha + \eps^\beta}$, 
$\mu(y) = \frac{\eps^\alpha}{\eps^\alpha + \eps^\beta}$. Thus 
$\lim_{\eps \to 0} \mu_\eps(x)$ depends very sensitively on 
the behavior of the elements of the transition matrix at small $\eps$. 

The reason for this is that the space of solutions to the defining equation 
$\mu_\eps P_\eps = \mu_\eps$ is one-dimensional in the case $\eps > 0$, but multidimensional
in the case $\eps = 0$. This also means that this linear equation is ill-conditioned
for small $\eps$. Thus computing $\mu_\eps(x)$ numerically by solving an eigenvalue problem 
is infeasible if the state space $S$ is large and the 
transition matrix is somewhat complicated. Metastability also means that a Monte 
Carlo simulation of $\mu_\eps$, i.e.\ running the chain $X^{(\eps)}$ and recording the 
relative occupation times of the states $x \in S$, will fail for small $\eps$. 
In the reversible case, the detailed balance equation $\mu_\eps(x) p_\eps(x,y) = 
\mu_\eps(y) p_\eps(y,x)$ can be used to compare the relative importance of 
$\mu_\eps(x)$ and $\mu_\eps(y)$ for neighboring $x,y \in S$, and by iterating for all 
$x,y \in S$, but there is no detailed balance equation for irreversible Markov chains. Therefore, 
it is not immediately clear how to compute the asymptotic stationary distribution
of an irreversible perturbed Markov chain in any numerically efficient way. 

Efficiently computing the stationary distribution of a large Markov chain is 
an extremely important problem in many areas of applied science. 
Maybe the most prominent example where
it is needed is the computation of the page rank in search engines \cite{LanMey}, where
metastability also plays a role. It is therefore not surprising that a 
large body of literature is devoted to the topic, mainly in the computer science 
community. The seminal paper here seems to be by Simon and Ando \cite{AnSi}, where they 
introduce a method for treating what is now known as almost decomposable Markov 
chains, and derive the metastable behavior and some information on the asymptotic 
stationary measure for such chains. 
Subsequently, the method was clarified and extended, and Meyer \cite{Mey}
realized that many of the extensions have a common foundation that he called the 
theory of the stochastic complement. Many further extensions and refinements of the method 
have been given since. We cannot give a full review of the literature here, but rather
point the reader, by way of example, 
to the recent papers \cite{Tif, ML11, dSte} and the references therein. 

An apparently independent effort to treat metastable Markov chains took place 
in the context of game theory and mathematical economy. 
Here the start was made by  HP Young \cite{Young}. 
He basically advocated using the Markov chain 
tree theorem, as given in 
\cite{Aldous/Fill} or \cite{FW98}. Up to normalization, 
it gives the stationary measure $\mu(x)$ as the sum of 
terms $w(t)$ indexed by the directed spanning trees of $S$ rooted in $x$, where 
the weight $w(t)$ of a tree $t$ is the product of all transition probabilities along its 
edges; for details see \cite{Aldous/Fill}. As has been pointed out in 
\cite{Ellison}, the problem with this formula is that while it is in principle not numerically unstable, it involves computing all
spanning trees, which is exponentially expensive and thus becomes 
non tractable for large state spaces. 
Moreover, all of the $w(t)$ are usually tiny, and so we are 
trying to add an astronomical number of tiny terms, which is not a good idea. 

A different approach was taken by Wicks and Greenwald \cite{WiGr1, WiGr2} who
offer a solution that is closer to the one described in \cite{Mey}, but differs in some
important details.  
At the center of their method is what they call the  
quotient construction on stochastic matrices, 
which allows them to recursively simplify the state space and, by keeping track of 
the various simplifications, to compute $\lim_{\eps \to 0} \mu_\eps$ in the end.  

As can already be guessed from the above discussion, the citation graph on  
metastable Markov chains and their stationary distributions is somewhat disconnected.
While some mathematicians, e.g.\ \cite{LaLo} or \cite{Deuf}, are aware of the theory 
of Ando and Simon \cite{AnSi}, it does not seem to be well known in the probability theory 
community. On the other hand, the mathematical theory of metastability following
\cite{BEGK1} is virtually unknown in the applied community, and the approaches by 
Young \cite{Young} and Wicks and Greenwald \cite{WiGr2} appear to be completely disjoint from 
the others. We hope that, among its other purposes, this paper helps connect these 
communities. For this reason we review the results related to Simon/Ando and those
of Wicks/Greenwald at the end of our paper, translate their statements from the 
language of matrices to probabilistic terminology, and comment on how their results relate 
to the present paper. 

The paper is organized as follows: in Section \ref{general escape}, 
we collect some results on escape times for irreducible Markov chains that seem hard to 
find in the literature. In Section \ref{asymptotic escape}, we introduce perturbed 
Markov chains and show how the results from Section \ref{general escape} 
can be used to obtain asymptotic 
expressions of various important quantities. These will be used in Section \ref{dynamics}
to describe the multi-scale effective dynamics of the chain. Finally, in Section 
\ref{algorithms}, we present our numerical algorithms and compare them to those
present in the computer science and economics literature.

\section{Stationary measures, escape probabilities and hitting distributions} 
\label{general escape}

Here we collect the main tools that we will use. 
In this section, $X$ is a general discrete time Markov chain. 
In contrast to the remainder of the paper, we do not assume the state space $S$ to be finite, 
but we will assume that $X$ is irreducible and recurrent unless 
stated otherwise. 

All of the results below 
are relatively transparent, explicit identities involving hitting times. 
Given the sheer amount of material on the subject, it is reasonable to assume that 
some or all of them have been derived elsewhere. We were unable to find an explicit 
reference for any of them, but will comment on related results where appropriate. 

For a Markov
chain $X$ on a state space $S$, 
the hitting time of a set $A \subset S$ is denoted by  
$\tau_A(X) = \inf \{ n \geq 0: X_n \in A \}$, and the return time
by $\tau_A^+(X) = \inf \{ n > 0: X_n \in A \}$. As usual, we will write
$\tau_x$ instead of $\tau_{\{x\}}$ for $x \in S$, 
and similarly for $\tau_x^+$. 

\begin{proposition}\label{basic fact}
Assume that $X$ is irreducible and positive recurrent, and write $\mu$ for the unique stationary distribution. 
Then for all $x, y \in S$,
\be \label{useful equation}
\mu(x)\bbP^x(\tau^+_y < \tau^+_x) = \mu(y)\bbP^y(\tau^+_x < \tau^+_y).
\ee
\end{proposition}

Proposition \ref{basic fact} looks like it should be part of every textbook on discrete time 
Markov chains, but somewhat surprisingly it is not. Before we comment on the status of Proposition
\ref{basic fact} in the literature, note that \eqref{useful equation} is 
reminiscent of the detailed balance equation. Let us write $q(x,y) := \bbP^x(\tau^+_y < \tau^+_x)$ and for the moment assume that 
$\sum_{y: y \neq x}  q(x,y) \leq 1$ for all $x \in S$. Then 
the quantities $q(x,y)$ can be completed to become the transition matrix of a 
{\em reversible}
Markov chain that has the same stationary distribution as $X$. In general, $\sum_{y \neq x}  q(x,y) \leq 1$
will not hold for all $x$, but below we will encounter a situation in the context of perturbed 
Markov chains where it does. 

When the Markov chain $X$ itself is reversible, 
Proposition \ref{basic fact} is a direct consequence of the 
well established theory of electrical networks: 
for example, from Proposition 9.5 in \cite{LPW} it follows that 
$\mu(x) \bbP^x(\tau_y^+ < \tau_x^+) = c_0 \caC(x \leftrightarrow y)$, where
 $\caC(x \leftrightarrow y)$ is the effective conductance between $x$ and $y$, and 
 $c_0$ a global constant.  
Since $\caC(x \leftrightarrow y)$ is symmetric in $x$ and $y$, 
\eqref{useful equation} follows in the reversible case.

For the non-reversible case, Proposition \ref{basic fact} appears much less well known, 
although it can also be quickly deduced from a known result:   
Corollary 8 of Chapter 2 in the unfinished, but brilliant, 
monograph by Aldous and Fill \cite{Aldous/Fill} directly implies it. 
As we have not found that statement anywhere else, we give a short proof here for the 
convenience of the reader. Our proof differs somewhat from the one given in 
\cite{Aldous/Fill} and uses the following more general lemma: 

\begin{lemma}\label{lem:ht}
Let $X$ be irreducible and positive recurrent. For all states $x, y, z \in S$, 
\[\mathbb{E}^z(\tau^+_x) = \mathbb{E}^z(\min(\tau^+_x,\tau^+_y)) + \p^z(\tau^+_y < \tau^+_x)\mathbb{E}^y(\tau^+_x).\]
\end{lemma}

\begin{proof}
Since the chain is irreducible and positive recurrent, $\mathbb{E}^z(\tau^+_x) < \infty$ for 
all $z$ and $x$ in the state space $S$, so in particular $\tau^+_y < \infty$ almost surely. 
Thus 
\begin{align*}
\mathbb{E}^z(\tau^+_x) & = \mathbb{E}^z(\tau^+_x,\tau^+_x \leq \tau^+_y) + \mathbb{E}^z(\tau^+_y+\tau^+_x-\tau^+_y,\tau^+_y < \tau^+_x)\\
	& = \mathbb{E}^z(\tau^+_x,\tau^+_x \leq \tau^+_y) + \mathbb{E}^z(\tau^+_y,\tau^+_y < \tau^+_x) + \mathbb{E}^z(\tau^+_x-\tau^+_y,\tau^+_y < \tau^+_x),
\end{align*}
The first two terms of the last line above sum up to $\mathbb{E}^z(\min(\tau^+_x,\tau^+_y))$. The last term is equal to 
$\mathbb{E}^z(\tau^+_x\circ\theta_{\tau^+_y}\cdot 1_{\{\tau^+_y < \tau^+_x\}})$, where 
$\theta_n X_j = X_{n+j}$ denotes the time shift by $n$ steps. 
Indeed, the random variables $(\tau^+_x-\tau^+_y)1_{\{\tau^+_y < \tau^+_x\}}$ and $\tau^+_x\circ\theta_{\tau^+_y}\cdot 1_{\{\tau^+_y < \tau^+_x\}}$, when nonzero, both count the number of steps from the first occurrence of $y$ until the first occurrence of $x$. By the strong Markov property,
\begin{align*}
\mathbb{E}^z(\tau^+_x\circ\theta_{\tau^+_y}\cdot 1_{\{\tau^+_y < \tau^+_x\}}) & = \mathbb{E}^z(\mathbb{E}^z(\tau^+_x\circ\theta_{\tau^+_y}\cdot 1_{\{\tau^+_y < \tau^+_x\}}\mid \mathcal{F}_{\tau^+_y}))\\
	& = \mathbb{E}^z(1_{\{\tau^+_y < \tau^+_x\}}\cdot \mathbb{E}^y(\tau^+_x)) = \p^z(\tau^+_y < \tau^+_x)\mathbb{E}^y(\tau^+_x),
\end{align*}
and the claim follows.
\end{proof}

\begin{proof}[Proof of Proposition \ref{basic fact}]
If $x = y$, the claim boils down to $0 = 0$, so let us assume that $x \neq y$. By using 
Lemma~\ref{lem:ht} in two different ways we obtain
\be \label{eq:prophsr}
\begin{split} 
\mathbb{E}^y(\tau^+_y) & = \mathbb{E}^y(\min(\tau^+_y,\tau^+_x)) + \p^y(\tau^+_x < \tau^+_y)\mathbb{E}^x(\tau^+_y), \\
\mathbb{E}^y(\tau^+_x) & = \mathbb{E}^y(\min(\tau^+_x,\tau^+_y)) + \p^y(\tau^+_y < \tau^+_x)\mathbb{E}^y(\tau^+_x).\\
\end{split}
\ee  
We rearrange the second equation above to obtain 
\[ 
\mathbb{E}^y(\min(\tau^+_x,\tau^+_y)) = \mathbb{E}^y(\tau^+_x) (1- \p^y(\tau^+_y < \tau^+_x)) = \mathbb{E}^y(\tau^+_x) \p^y(\tau^+_x < \tau^+_y),
\]
the last equality being due to  $x \neq y$. Plugging this back into the first line of 
\eqref{eq:prophsr}, using the fact that $\mu(y) = \mathbb{E}^y(\tau^+_y)^{-1}$, and 
rearranging, gives  
\be \label{lemma8}
\mu(y) \bbP^y(\tau^+_x < \tau^+_y) = \frac{1}{\mathbb{E}^x(\tau^+_y) + \mathbb{E}^y(\tau^+_x)},
\ee
which is essentially Corollary 8 of Chapter 2 of \cite{Aldous/Fill}. For our purposes, we note that 
the right-hand side of \eqref{lemma8} is invariant under swapping $x$ and $y$, which proves
the claim.
\end{proof}

{\bf Remark:} In the continuous time setting, the whole proof of Lemma \ref{lem:ht} and 
almost all of the proof of Proposition \ref{basic fact} goes through unchanged if we define
$\tau_x^+ = \inf \{t > 0: X_t = x, X_s \neq X_0 \text{ for some } 0 \leq s \leq t \}.$ The only 
difference is that the formula for the stationary measure in that case is given by 
$\mu(y) = \bbE^y(\tau_y^+)^{-1} \lambda(y)^{-1}$, where $\lambda(y)$ is 
the exponential rate with which the process jumps away from $y$. 
This gives the formula
$$ \mu(x) \lambda(x) \bbP^x(\tau_y^+ < \tau_x^+) = \mu(y) \lambda(y) \bbP^y(\tau_x^+ < 
\tau_y^+),$$
which is a special case of the symmetry result on capacities for non-reversible 
continuous time Markov chains derived by Gaudilli\'ere and Landim \cite{GaLa11}, 
and applied to investigate metastability by Beltr\'an and Landim \cite{BeLa12}. 
Their proof is quite different from the one presented here. 

A direct consequence of  Proposition \ref{basic fact} is
\begin{corollary} \label{stat dist representation}
The stationary distribution $\mu$ of $X$ fulfills the set of equations   
\be \label{generic u expression}
\frac{1}{\mu(x)} = \sum_{y \in S} \frac{\bbP^x(\tau_y^+ \leq \tau_x^+)}
{\bbP^y( \tau_x^+ \leq \tau_y^+)}.
\ee
\end{corollary}
\begin{proof}
Since $\{\tau_x^+ = \tau_y^+\} = \emptyset$ if 
$x \neq y$, \eqref{useful equation} is equivalent to 
$$
\mu(x) \bbP^x(\tau^+_y \leq \tau^+_x) = \mu(y) \bbP^y(\tau^+_x \leq \tau^+_y)
$$
for all $x,y \in S$. 
We have $\bbP^y(\tau^+_x \leq \tau^+_y) > 0$ by irreducibility for all $x,y \in S$, 
and so we can divide both sides by it. Summing over $y \in S$ and rearranging 
now shows the claim. 
\end{proof}

To get the most out of Corollary \ref{stat dist representation}, we need find 
a way to compute the escape probabilities appearing in \eqref{generic u expression}. We will
now collect some tools that will help us to do this, asymptotically, in the context of perturbed Markov chains.  
Unlike the statement of Proposition \ref{basic fact}, we have not been able to find them in the 
literature, but we still suspect that they are not completely new. 

\begin{proposition} \label{first formula}
Let $X$ be an irreducible, recurrent Markov chain. For $A \subset S$, $x \in S$ and $y \in A$, we have 
\be \label{hitting prob 1a}
\bbP^x(X_{\tau_A^+} = y) = p(x,y) + \sum_{z \in S \setminus A} 
\frac{\bbP^x(\tau_z^+ < \tau_A^+)}{\bbP^z(\tau_A^+ < \tau_z^+)} p(z,y).
\ee
When $x \in S \setminus A$, \eqref{hitting prob 1a} simplifies to 
\be \label{hitting prob 1b}
\bbP^x(X_{\tau_A^+} = y) = \sum_{z \in S \setminus A} 
\frac{\bbP^x(\tau_z < \tau_A^+)}{\bbP^z(\tau_A^+ < \tau_z^+)} p(z,y).
\ee
\end{proposition}

\begin{proof}
For $z \in S \setminus A$, let us write $\Omega_{y,k,z}$ for the set of paths 
that visit $z$ precisely $k$ times before entering $A$, and in addition
move directly from $z$ to $y \in A$. More formally, we put $\tau_{z,0}^+ := 0$, and  
\[
\tau_{z,k}^+ := \min \{ n \in \bbN: | \{ 0 < j \leq n: X_j = z \}| = k \},
\]
for $k \geq 1$, where $|.|$ denotes the cardinality of a set in this case. Then, 
\[
\Omega_{y,k,z} := \{ \tau_{z,k}^+ < \tau_A^+, X_{\tau_{z,k}^+ +1} = y \}.
\]
We have 
$\bigcup_{k \geq 1} \bigcup_{z \in S \setminus A} \Omega_{y,k,z} = 
\{ 1 < \tau_A^+ < \infty, X_{\tau_A^+} = y\}$, and 
the sets $\Omega_{y,k,z}$ are disjoint.  
As the chain is irreducible and recurrent, $\bbP(\tau_A^+ = \infty) = 0$ holds, and thus  
%$\bbP^x(\Omega_{y,k,z}) \leq \bbP^x(\tau_{z,k}^+ < \tau_A^+) \leq (1 - \delta)^{k-1}$ for 
%$\delta = \bbP^z(\tau_A^+ < \tau_z^+) > 0$, by the strong Markov property. Thus 
\be \label{prelim 1}
\bbP^x(X_{\tau_A^+} = y) = p(x,y) + \sum_{z \in S \setminus A} \sum_{k \geq 1} 
\bbP^x(\Omega_{y,k,z}).
\ee
Now for $k \geq 1$ we compute 
\[
\begin{split}
\bbP^x(\Omega_{y,k,z}) & = \bbE^x ( \bbP^x ( \tau_{z,k}^+ < \tau_A^+, \tau_z^+ < \tau_A^+, 
X_{\tau_{z,k}^+ +1} = y | \caF_{\tau_z^+})) = \\
& = \bbP^x(\tau_z^+ < \tau_A^+) \bbP^z(\tau_{z,k-1}^+ < \tau_A^+, X_{\tau_{z,k-1}^+ +1} = y ) \\ 
&=  \bbP^x(\tau_z^+ < \tau_A^+) \bbP^z(\Omega_{y,k-1,z}) = \bbP^x(\tau_z^+ < \tau_A^+) 
\bbP^z(\tau_z^+ < \tau_A^+)^{k-1} \bbP^z(\Omega_{y,0,z}).
\end{split}
\]
In the second line, we used the strong Markov property, and in the third line, 
finite induction. We now sum up the geometric series in $k$, use $\bbP^z(\Omega_{y,0,z}) = p(z,y)$, and obtain
\[
\sum_{k \geq 1} \bbP^x(\Omega_{y,k,z}) =  \frac{\bbP^x(\tau_z^+ < \tau_A^+)}{1 - 
\bbP^z(\tau_z^+ < \tau_A^+)} p(z,y) = \frac{\bbP^x(\tau_z^+ < \tau_A^+)}{\bbP^z(\tau_A^+ < 
\tau_z^+)} p(z,y). 
\]
In the last equality we used that $z \notin A$ implies $\bbP^z(\tau_z^+ = \tau_A^+)=0$. 
Plugging this into \eqref{prelim 1} proves \eqref{hitting prob 1a}.

For \eqref{hitting prob 1b}, let us start from \eqref{hitting prob 1a} and note that 
%for all $x \in S \setminus A$, we have 
%$\bbP^x(X_{\tau_A^+} = y) = \bbP^x(X_{\tau_A} = y)$; 
for $z \neq x$, we have 
$\bbP^x(\tau_z^+ < \tau_A^+) = \bbP^x(\tau_z < \tau_A^+)$. For the term with $z = x$,
we have 
\[
p(x,y) + \frac{\bbP^x(\tau_x^+ < \tau_A^+)}{\bbP^x(\tau_A^+ < \tau_x^+)} p(x,y) = 
p(x,y) \frac{1}{\bbP^x(\tau_A^+ < \tau_x^+)} = p(x,y) 
\frac{\bbP^x(\tau_x < \tau_A^+)}{\bbP^x(\tau_A^+ < \tau_x^+)}.
\]
The first equality holds because for $x \notin A$, $\bbP^x(\tau_A^+ < \tau_x^+) + 
\bbP^x(\tau_A^+ < \tau_x^+) = 1.$ Thus \eqref{hitting prob 1b} is shown.
\end{proof}

A variant of Proposition \ref{first formula} is well known and is the basis of many 
algorithms for computing stationary distributions of large Markov chains. 
It is called the quotient construction by Wicks and Greenwald
\cite{WiGr1,WiGr2}, and the stochastic complement by Meyer \cite{Mey}. 
While in all those references, it is written in matrix language, we give here the probabilistic
formulation, which also has the benefit that we can give a short and transparent proof. 
Below and in what follows $A^c$ denotes the complement of a set $A$. 

\begin{proposition} \label{quotient}
(\cite{Mey, WiGr2}) Assume that the state space $S$ is finite, $A \subset S$, $x \in S$ and $y \in A$. Then  
\be \label{quot}
\bbP^x (X_{\tau_A^+} = y) = p(x,y) + \sum_{z,w \in A^c} p(x,w) (1 - P|_{A^c})^{-1}(w,z) p(z,y),
\ee
where $P|_{A^c} = (p(x,y))_{x,y \in A^c}$ is the restriction of the transition matrix $P$ to $A^c$. 
\end{proposition}

\begin{proof}
Clearly, $\bbP^x (X_{\tau_A^+} = y) = p(x,y) + \sum_{w \in A^c} p(x,w) \bbP^w(X_{\tau_A} = y)$.
Now, standard results \cite{LPW} state that $h_y(w) := \bbP^w(X_{\tau_A} = y)$ is the unique 
harmonic extension of the function $1_{\{y\}}$ from $A$ to $S$. In other words, $h_y$ is the unique 
function so that $P h_y(w) = h_y(w)$ for all $w \in A^c$, and $h_y(w) = 1_{\{y\}}(w)$ on $A$. This can 
be rewritten as $(P|_{A^c} - 1) h_y(w) = - P 1_{\{y\}}(w) = -p(w,y)$ for all $w \in A^c$. 
Since $X$ is irreducible, there exists $n \in \bbN$ with $\| (P|_{A^c})^n \| < 1$, where 
$\|. \|$ is the operator norm of a matrix. Thus $(1-P|_{A^c})$ is invertible. 
The claim follows. 
\end{proof}

{\bf Remark:} 
In \eqref{quot}, the probability of 
the set of all paths moving from $w$ to $z$ in $A^c$ and then entering $A$ from there is expressed as  
$(1 - P|_{A^c})^{-1}(w,z)$. In \eqref{hitting prob 1a} the probability of the set of all paths that leave $A^c$ at $z$ 
but enter anywhere is expressed as the quotient of two escape probabilities. Comparing the two and varying over $p(z,y)$ leads to the amusing identity 
\[
\sum_{w \in A^c} p(x,w) (1 - P|_{A^c})^{-1}(w,z) = \frac{\bbP^x(\tau_z^+ < \tau_A^+)}{\bbP^z(\tau_A^+ < \tau_z^+)},
\]
for all $A \subset S$, $x \in S$ and $z \in A^c$. \\[1mm]

For the following result, we do not assume irreducibility of the chain. 

\begin{lemma} \label{committor formula}
Let $(X_n)$ be an arbitrary Markov chain, $A,B \subset S$. Assume 
$x \notin A \cup B$ and 
$\bbP^x(\tau_B^+ < \infty) > 0$. Then
\[
\bbP^x(\tau_B^+ < \tau_A^+) = \frac{\bbP^x(\tau_B^+ < \tau^+_{A \cup \{x\}})}{\bbP^x(\tau_B^+ < \tau_x^+)}.
\]
\end{lemma}

\begin{proof}
We have
\[ 
\begin{split}
\bbP^x(\tau_B^+ < \tau_A^+) &= \bbP^x(\tau_B^+ < \tau^+_{A \cup \{x\}}) + \bbP^x(\tau^+_x < \tau^+_B < 
\tau^+_A)\\
& = \bbP^x(\tau^+_B < \tau^+_{A \cup \{x\}}) + \bbP^x(\tau^+_x < \tau^+_B)
\bbP^x(\tau_B < \tau_A),
\end{split}
\]
where in the last step we have used the strong Markov property. 
By our assumption $x \notin A \cup B$, we have $\bbP^x(\tau_B < \tau_A) = \bbP^x(\tau_B^+ < \tau_A^+)$.
Since we assumed $\bbP^x(\tau_B^+ < \infty) > 0$, we must have $1 -  
\bbP^x(\tau^+_x < \tau^+_B) =  \bbP^x(\tau^+_B < \tau^+_x) > 0$; otherwise the 
strong Markov property would give $\bbP^x(\tau_B^+ < \infty) = 0$. 
Thus we can rearrange and obtain the result.
\end{proof}

For our next statement, 
fix a proper subset $C \subsetneq S$, and define for all $x,y \in S$ 
\be \label{effective p}
\tilde p(x,y) := p(x,y) \text{ if } x \notin C, \qquad 
\tilde p(x,y) := \bbP^x(X_{\tau_{C^c}} = y) \text{ if } x \in C.
\ee

\begin{proposition} \label{effective chain}
Let $X$ be an irreducible, recurrent Markov chain. Then 
$\tilde P = (\tilde p(x,y))_{x,y \in S}$ is the transition matrix of a Markov chain 
$\tilde X$.  Denoting its  path measure by $\tilde \bbP$, we have
\be \label{delete C}
\tilde \bbP^x(\tau_B < \tau_A) = \bbP^x(\tau_B < \tau_A).
\ee
for all  $A,B \subset S$ with $(A \cup B) \cap C = \emptyset$, and all $x \in S$. 
\end{proposition}

\begin{proof}
Since $(X_n)$ is irreducible and recurrent and  $C^c \neq \emptyset$, $\bbP^x(\tau_{C^c} < \infty) = 1$ for all
$x \in C$. Thus it is obvious that $\tilde P$ is a stochastic matrix. The statement
\eqref{delete C} is also intuitively obvious, since all we do is replace the motion inside
$C$ with the effective motion from $C$ to its exterior. We nevertheless give the short
formal proof. 

We write $\sigma_m$ for the $m$-th time that the chain $(X_n)$ travels between two 
states that are not both in $C$, i.e.
$$
\sigma_0 := 0, \quad \sigma_m := \min \{ n > \sigma_{m-1}: X_n \notin C \text{ or } 
X_{n-1} \notin C \}.
$$
On $\Omega_0 = \{ \sigma_m < \infty \, \, \forall m \in \bbN \}$, 
we define $\tilde X_m = X_{\sigma_m}$. Then $\bbP^x(\Omega_0)=1$ for all $x \in S$ by
recurrence and irreducibility of $X$, and $\tilde X$ is a Markov chain by the 
strong Markov property of $X$. Since 
$\bbP^x(\tilde X_{1} = y)  = \bbP^x(X_{\sigma_{1}} = y) = \tilde p(x,y)$, 
the transition probabilities of $\tilde X$ are given by 
\eqref{effective p}.
Since $C$ is disjoint from $A$ and $B$, we have 
$$\{ \tau_A(X) < \tau_B(X) \} \cap \Omega_0 = 
\{ \tau_A(\tilde X) = \tau_B(\tilde X) \} \cap \Omega_0,$$
and  \eqref{delete C} follows
by taking expectations.
\end{proof}

%When $C$ contains just one point, say $x$, we explicitly have 
%$$
%\tilde p(x,y) = \frac{1}{\sum_{w \in S, w \neq x} p(x,w)} p(x,y), \quad \tilde p(x,x) = 0.
%$$
%In this case, we can say a bit more:
%
%\begin{proposition} \label{scaling}
%In the situation of Proposition \ref{effective chain}, assume that $C = \{x\}$. Then
%for all $y \in S$,  
%$$
%\bbP^x( \tau_y^+ < \tau_x^+) = \sum_{w \in S, w \neq x} p(x,w) \tilde \bbP^x( \tau_y^+ < \tau_x^+).
%$$
%\end{proposition}
%
%\begin{proof}
%We have 
%$$
%\bbP(\tau_y^+ < \tau_x^+) = \sum_{q \in S, q \neq x} p(x,q) \bbP^q(\tau_y < \tau_x) =
%\sum_{w \in S, w \neq x} p(x,w) \sum_{q \in S, q \neq x} \tilde p(x,q)
%\bbP^q(\tau_y < \tau_x)$$
%Now $\tilde p(x,x) = 0$, and $\bbP^q(\tau_y < \tau_x) = \tilde \bbP^q(\tau_y < \tau_x)$
%since on $\{\tau_y < \tau_x\}$ no transition from $x$ to another state occurs. The claim 
%follows.  
%\end{proof}

For our final general statement, we introduce the notion of a direct path which will
be useful in several places below. 
Let $J$, $A$ and $B$ be subsets of $S$. 
A tuple $\gamma = (x_1, \ldots, x_n) \in S^n$ 
is called a {\em direct $J$-path of length $n$ from $A$ to $B$} if $x_1\in A$, $x_n \in B$, and for all $1\leq i < j\leq n$, if $x_i=x_j$ then $i=1$ and $j=n$. Note that we allow $x_1,x_n \notin J$. The set of all direct $J$-paths from 
$A$ to $B$ will be denoted by $\Gamma_J(A,B)$, and the components of $\gamma \in 
\Gamma_J(A,B)$ will be written $\gamma_i$, $i=1,\ldots,n$. $|\gamma|$ will denote the length of $\gamma$. 
For $A = \{x\}$ or $B = \{y\}$ we will use the notations $\Gamma(A,y)$ instead of $\Gamma(A,\{y\})$ etc, and 
speak of direct $J$-paths from $A$ to $y$, from $x$ to $y$ or from $x$ to $B$. The probability of a direct 
$J$-path is defined by $\bbP(\gamma) := \prod_{j=1}^{|\gamma|-1} p(\gamma_{j},\gamma_{j+1})$.

\begin{proposition} \label{second formula}
Let $J$ be a finite subset of $S$. Then for all $x \in J$ and $y \in S \setminus J$,
\be \label{hitting prob 2}
\p^x(X_{\tau_{S \setminus J}} = y) = \sum_{\gamma \in \Gamma_J(x,y)} 
\prod_{i=1}^{|\gamma| -1} \frac{p(\gamma_i,\gamma_{i+1})}{1-\p^{\gamma_i}(X_{\tau^+_{(S\backslash J)\cup \{\gamma_1,\dots,\gamma_i\}}} = \gamma_i)}
\ee
\end{proposition}

\begin{proof}
The idea of the proof is to start at state $x$ and run the Markov chain until it either hits $S\setminus J$ or returns to $x$. In the first case we have reduced the problem to computing $\bbP^z(X_{\tau_{(S\setminus J)\cup\{x\}}} = y)$ and we iterate the argument for the smaller set $J\setminus\{x\}$; in the second case we use the strong Markov property to restart the process. Formally, let us proceed by induction on $|J|$. The claim trivially holds for $J = \emptyset$; so now let $x \in J$. The third equality below is obtained by the strong Markov property.
\begin{align*}
\bbP^x(X_{\tau_{S \setminus J}} = y) & = \bbP^x(X_{\tau_{S\setminus J}} = y, 
\tau_{S\setminus J} < \tau^+_{x}) + \bbP^x(X_{\tau_{S\setminus J}} = y,
\tau^+_{x} < \tau_{S\setminus J}) \\
	& =\bbP^x(X_{\tau^+_{(S\setminus J)\cup\{x\}}} = y) + \bbE^x(1_{\{ \tau^+_{x} < \tau_{S\setminus J} \} } \bbP^{X_{\tau^+_x}}(X_{\tau_{S\setminus J}} = y))\\
	& =\bbP^x(X_{\tau^+_{(S\setminus J)\cup\{x\}}} = y) + \bbP^x(X_{\tau^+_{(S\setminus J)\cup\{x\}}} = x) \bbP^x(X_{\tau_{S\setminus J}} = y)
\end{align*}
As the Markov chain is recurrent and irreducible, 
we have $\bbP^x(X_{\tau^+_{(S\setminus J)\cup\{x\}}} = x) < 1$.
Thus the last equation can be rearranged to 
\[
\bbP^x(X_{\tau_{S\setminus J}} = y) = \frac{\bbP^x(X_{\tau^+_{(S\setminus J)\cup\{x\}}} = y)}{1-\bbP^x(X_{\tau^+_{(S\setminus J)\cup\{x\}}} = x)}
\]
where the numerator may be decomposed as $p(x,y) + \sum_{z\in J\setminus\{x\}}p(x,z)  \bbP^z(X_{\tau_{(S\setminus J)\cup\{x\}}} = y)$.

Finally, we use the induction hypothesis for the set $J \setminus \{x\}$ to rewrite 
$\bbP^z(X_{\tau_{(S\setminus J) \cup\{x\}}} = y)$ for all $z \in J\setminus\{x\}$, and obtain
\begin{align*}
& \bbP^x(X_{\tau_{S\setminus J}} = y)  = \frac{p(x,y)}
{1-\bbP^x(X_{\tau^+_{(S\setminus J)\cup\{x\}}} = x)}\\
	& + \sum_{z \in J \setminus \{ x \}} \sum_{ \gamma\in \Gamma_{J\setminus\{x\}}(z,y)} 
	\frac{p(z,y)} {1-\bbP^x(X_{\tau^+_{(S\setminus J)\cup\{x\}}} = x)} \prod_{i=1}^{|
	\gamma|-1} \frac{p(\gamma_i,\gamma_{i+1})}{1-\p^{\gamma_i}(X_{\tau^+_{(S\setminus J)\cup 
	\{x,\gamma_2,\dots,\gamma_i\}}} = \gamma_i)}
\end{align*}
Re-indexing yields the claim.
\end{proof}

\section{Perturbed Markov chains: escape probabilities}
\label{asymptotic escape}

Let $X^{(0)} =(X_n^{(0)})_{n \in \bbN}$ be a Markov chain 
on a finite state space $S$. 
A family $X^{(\eps)} = (X^{(\eps)}_n)_{n \in \bbN}$ of Markov 
chains on $S$ indexed by $\eps \geq 0$ is called a {\em perturbation} 
of $X^{(0)}$ if $\lim_{\eps \to 0} p_\eps(x,y) = p_0(x,y)$ for all $x,y \in S$, 
where $p_\eps(x,y)$ denotes the elements of the transition matrix $P_\eps$ 
of the chain $X^{(\eps)}$, $\eps \geq 0$. We will speak of an 
{\em irreducible perturbation} of $X^{(0)}$ (or, alternatively, call the 
family $X^{(\eps)}$ an {\em irreducibly perturbed Markov chain})
if the chain $X^{(\eps)}$ is irreducible for all $\eps > 0$. 

Note that in the definition of irreducibly perturbed Markov chains, we do not require that 
$X^{(0)}$ be irreducible, and indeed the case where $X^{(0)}$ has several 
ergodic components is the interesting one. 
Recall that $x \in S$ is called {\em accessible} from $y \in S$ under $X^{(\eps)}$
if $\bbP_\eps^y(X_n = x) > 0$ for some $n \geq 0$. We write $x \to y$ if $y$ is accessible 
from $x$, and say that two states $x$ and 
$y$ {\em communicate} if $x \to y$  and $y \to x$. 
The property to communicate forms an equivalence relation, and the respective equivalence
classes are called {\em communicating classes}. A state $x$ is called {\em essential} 
if $y \to x$ for all $y \in S$ such that $x \to y$, otherwise transient.
It is easy to see that either all
members of a communicating class $E$ are essential, or all are transient. 
In the first case, $E$ is called
an essential (communicating) class, or ergodic component.

$S$ can thus be decomposed into finitely many disjoint
essential classes $E_1, \ldots E_n$ and the set 
$F = S \setminus \bigcup_{i=1}^n E_i$ of transient states.
To emphasize that a nontrivial ergodic decomposition only exists for $\eps = 0$, 
we will always speak of $P_0$-essential classes and $P_0$-transient states.  
$\caE$ will denote the set of all $P_0$-essential classes. 

The sets $E_i$ and $F$ 
can be conveniently described in terms direct paths. The following 
statement could be taken as a definition of $P_0$-essential classes and $P_0$-transient 
states; the proof of
equivalence to the traditional definition of essential classes (see e.g.\ \cite{LPW}) 
is very easy, and omitted here. Here and below, we will say that a direct path $\gamma$ is 
{\em$P_0$-relevant} if $\bbP_0(\gamma) = \lim_{\eps \to 0} \bbP_\eps(\gamma) > 0$, otherwise 
$P_0$-irrelevant.

\begin{lemma} \label{direct path lemma}
Let $X^{(\eps)}$ be an irreducibly perturbed Markov chain. \\
a)  $x,y \in S$ are in the same $P_0$-essential class $E$ if and only if there exists
a $P_0$-relevant direct $E$-path from $x$ to $y$, and a $P_0$-relevant 
direct $E$-path from $y$ to $x$. \\
b) $x \in S$ is in the transient set $F$ if and only if all
direct $S$-paths from $\bigcup_{j=1}^n E_j$ to $x$ are $P_0$-irrelevant. 
\end{lemma}

In much of what follows, we will use the following concept of asymptotic equivalence. 
Two functions $\eps \mapsto a_\eps$ and 
$\eps \mapsto b_\eps$ from $\bbR_0^+$ to $\bbR_0^+$ are {\em asymptotically equivalent},
if either $a_\eps$ and $b_\eps$ are identically zero, 
or $b_\eps>0$ for all $\eps \in (0,\eps_0)$ with some $\eps_0>0$ and 
$\lim_{\eps \to 0} a_\eps / b_\eps = 1$. Note that in the latter case, we do not assume convergence of 
$a_\eps$ or $b_\eps$. We write  $a_\eps \simeq b_\eps$ if $a_\eps$ is asymptotically equivalent to $b_\eps$. 
It is easy to see that $\simeq$ is indeed an 
equivalence relation, and in particular this implies  
$1/a_\eps \simeq 1/b_\eps$ whenever $a_\eps \simeq b_\eps$ and 
$a_\eps$ is not identically zero. We will 
also need to know that $\simeq$ is stable under addition 
and multiplication in the 
following sense: if $a_\eps \simeq b_\eps$ and $c_\eps \simeq d_\eps$,
then $a_\eps + c_\eps \simeq b_\eps + d_\eps$, and 
$a_\eps c_\eps \simeq b_\eps d_\eps$. Stability under multiplication is trivial, and stability under addition follows from 
\[
|\tfrac{a_\eps + c_\eps}{b_\eps + c_\eps} - 1| = 
| \tfrac{a_\eps/b_\eps-1}{1 + c_\eps/b_\eps} | \leq |\tfrac{a_\eps}{b_\eps} - 1|
\] 
and transitivity of $\simeq$. Note that we did not assume that $a_\eps \simeq c_\eps$ in either case. 

Let $E \in \caE$ be a $P_0$-essential class. The restriction of $X^{(0)}$ to $E$ 
is the Markov chain with
state space $E$ and transition matrix $(p_0(x,y))_{x,y \in E}$. It 
is irreducible, and thus has a unique strictly positive stationary distribution $\nu_E$. 
The trivial extension of $\nu_E$ to $S$ (by putting $\nu_E(x) := 0$ for $x \notin E$)
will be denoted by the same symbol, and is an extremal point of the convex set  
of stationary distributions for $P_0$. 
The following lemma shows that when we focus our attention on a single 
$P_0$-essential class, the unperturbed chain gives a faithful asymptotic description of 
both the dynamics and the stationary distribution. Here and below we will write 
$\mu_\eps$ for the unique stationary distribution of $X^{(\eps)}$,  when $0 < \eps$.

\begin{lemma} \label{like unperturbed}
Let $E \in \caE$ be a $P_0$-essential class. Then for all $x,y \in E$ and all 
$z \in S$,  
\be \label{unperturbed1}
\lim_{\eps \to 0} \bbP^x_\eps( \tau_y^+ < \tau_z^+) = \bbP_0^x( \tau_y^+ < \tau_z^+), \quad \text{and} \quad \lim_{\eps \to 0} \bbP^x_\eps( \tau_y < \tau_z) = \bbP_0^x( \tau_y < \tau_z),
\ee
and 
\be \label{unperturbed 2}
\lim_{\eps \to 0} \frac{\mu_\eps(x)}{\mu_\eps(y)} = \frac{\nu_E(x)}{\nu_E(y)}.
\ee
In particular, $\mu_\eps(x) / \mu_\eps(y) \simeq \nu_E(x) / \nu_E(y)$.
\end{lemma}

\begin{proof}
We only prove the first equality from \eqref{unperturbed1}, 
the proof for the second one is identical. We decompose
\be \label{decomp}
\bbP_\eps^x(\tau_y^+ < \tau_{z}^+) = \bbP_\eps^x(\tau_y^+ < \tau_{\{z\} \cup E^c}^+) + 
\bbP_\eps^x(\tau_{E^c} < \tau_y^+ < \tau_z^+).
\ee
The second term is equal to $\sum_{w \in E^c} \bbP_\eps^x(X_{\tau_{E^c \cup \{y\}}^+} = w) 
\bbP_\eps^w(\tau_y^+ < \tau_z^+)$, and Proposition \ref{first formula} gives 
$$
\bbP_\eps^x(X_{\tau_{E^c \cup \{y\}}^+} = w) = p_\eps(x,w) + \sum_{u \in E \setminus \{y\}}
\frac{\bbP_\eps^x(\tau_u^+ < \tau_{E^c \cup \{y\}})}{\bbP_\eps^u(\tau_{E^c \cup \{y\}}^+ 
< \tau_u^+)} p_\eps(u,w).
$$
Now for each $u \in E$, there is a $P_0$-relevant 
direct $E$-path $\gamma$ from $u$ to $y$, and so 
$$
\liminf_{\eps \to 0} \bbP_\eps^u(\tau_{E^c \cup \{y\}}^+ 
< \tau_u^+) \geq \liminf_{\eps \to 0} \bbP_\eps^u(\tau_{y}^+ 
< \tau_u^+) \geq \lim_{\eps \to 0} \bbP_\eps^u(\gamma) > 0.
$$
Thus $\lim_{\eps \to 0} \bbP_\eps^x(X_{\tau_{E^c \cup \{y\}}^+} = w) = 0$, and thus 
the second term on the right-hand side of \eqref{decomp} vanishes as $\eps \to 0$. 

For the first term of \eqref{decomp}, fix $n \in \bbN$ and decompose 
$$
\bbP_\eps^x(\tau_y^+ < \tau_{\{z\} \cup E^c}^+) = \bbP_\eps^x(n \leq \tau_y^+ < \tau_{\{z\} \cup E^c}^+) + \bbP_\eps^x(\tau_y^+ < \tau_{\{z\} \cup E^c}^+, \tau_y^+ < n). 
$$
We have  
\[
\lim_{\eps \to 0} \bbP^x_\eps(\tau_y^+ < \tau_{\{z\} \cup E^c}, \tau_y^+ < n)  = 
\bbP_0^x( \tau_y^+ < \tau_{\{z\} \cup E^c}, \tau_y^+ < n) = \bbP_0^x (\tau_y^+ < \tau_{z}^+, \tau_y^+ < n).
\]
The first equality is because the probability on the left-hand side is a finite sum of 
at most $n$-fold products of transition
probabilities. The elementary Markov property at time $m < n$ gives 
\[
\bbP_\eps^x(n\leq\tau_y^+\leq \tau_{\{z\} \cup E^c}^+) \leq \bbP_\eps^x (\tau_y^+ \geq m) 
\sup_{w \in E} \bbP^w_\eps (n-m \leq \tau_y^+ \leq \tau_{\{z\} \cup E^c}^+).
\]
For each $w \in E$, there is a $P_0$-relevant direct $E$-path $\gamma$ from $w$ to $y$, 
and thus there exists $c>0$ with $\bbP_\eps^w(\tau_y^+ \geq |E|) \leq 
(1-c)$ for all $\eps$ sufficiently small. We conclude that $\bbP_\eps^x(n \leq \tau_y^+ < \tau_{\{z\} \cup E^c}^+) \leq (1-c)^{\lfloor n/|E| \rfloor}$ for 
all sufficiently small $\eps \geq 0$ and thus 
$$
\limsup_{\eps \to 0} | \bbP_\eps^x (\tau_y^+ < \tau_{\{z\} \cup E^c}^+) - \bbP_0^x
(\tau_y^+ < \tau_z^+) | \leq 2 (1-c)^{\lfloor n/|E| \rfloor}.
$$
As $n$ was arbitrary, \eqref{unperturbed1} follows. For \eqref{unperturbed 2}, we apply 
\eqref{useful equation} and obtain 
$$
\frac{\mu_\eps(x)}{\mu_\eps(y)} = \frac{\bbP_\eps^y(\tau_x^+ < \tau_y^+)}
{\bbP_\eps^x(\tau_y^+ < \tau_x^+)} \stackrel{\eps \to 0}{\longrightarrow} \frac{\bbP_0^y(\tau_x^+ < \tau_y^+)}
{\bbP_0^x(\tau_y^+ < \tau_x^+)} = \frac{\nu_E(x)}{\nu_E(y)}.
$$
Since the right-hand side above is strictly positive, this implies 
$\mu_\eps(x) / \mu_\eps(y) \simeq \nu_E(x) / \nu_E(y)$.
\end{proof}

As an immediate corollary, we obtain some information on the structure of the stationary 
distribution in the limit $\eps \to 0$. Recall that $\caE = \{E_1, \ldots, E_n\}$ is the 
collection of $P_0$-essential classes, and $F$ is the set of transient states. 

\begin{corollary} \label{structure} Let $z \in S$. \\
a) If $z \in F$, then $\lim_{\eps \to 0} \mu_\eps(z) = 0$.\\
b) If $z \in E$ for some $E \in \caE$, then 
$\mu_\eps(z) \simeq \mu_\eps(E) \nu_E(z)$.\\
c) In particular if $\lim_{\eps \to 0} \mu_\eps(E)$ exists for all $E \in \caE$, then
$\lim_{\eps \to 0} \mu_\eps(x)$ exists for all $x \in S$, and  
$$ \lim_{\eps \to 0} \mu_\eps(x) = \sum_{E \in \caE} \lim_{\eps \to 0} \mu_\eps(E) \nu_E(x).$$ 
\end{corollary}

\begin{proof}
For all $z \in F$, there exists a $P_0$-relevant path from $z$ to $\cup_{j=1}^n E_j$, 
and thus there is $x \in \cup_{j=1}^n E_j$ with 
$\liminf_{\eps \to 0} \bbP_\eps^z(\tau_x^+ < \tau_z^+) > 0$. On the other hand, by 
\eqref{unperturbed1},  
$\lim_{\eps \to 0} \bbP_\eps^x(\tau_z^+ < \tau_x^+) = 0$. So by Proposition 
\ref{basic fact}, 
$$
\lim_{\eps \to 0} \mu_\eps(z) = \lim_{\eps \to 0} \bbP_\eps^z(\tau_x^+ < \tau_z^+)^{-1} \bbP_\eps^x(\tau_z^+ < \tau_x^+) \mu_\eps(x) = 0,
$$ 
proving a). 
Now let $z \in E$ for some $P_0$-essential class $E$. 
Summing the asymptotic equality $\mu_\eps(y) \nu_E(z) \simeq \mu_\eps(z) \nu_E(y)$ over 
$y \in E$ gives b), and c) is immediate from a) and b). 
\end{proof}

The practical usefulness of Corollary \ref{structure} depends on our ability to compute
asymptotic expressions for the $\mu_\eps(E)$. 
We now give two statements that will play a key role in all that follows. 
The first says that hitting probabilities are asymptotically 
equivalent when the transition matrices are. The second  describes how a perturbed 
Markov chain leaves a $P_0$-essential class,
with or without the additional condition that it cannot return to its starting point.

\begin{theorem} \label{chain comparison}
Let $X^{(\eps)}$ and $\tilde X^{(\eps)}$ be perturbed 
Markov chains with finite state space $S$, 
but not necessarily irreducible. Let us assume that 
$p_\eps(x,y) \simeq \tilde p_\eps(x,y)$ 
for the elements of the respective transition matrices. Then 
for all $A,B \subset S$ and all $x \in S$, we have 
$$\bbP^x_\eps(\tau_B < \tau_A) \simeq \tilde \bbP_{\eps}^x(\tau_B < \tau_A).$$
\end{theorem}

\begin{proof}
We will first show that the statement holds
in the case where $P_\eps$ and $\tilde P_\eps$ only differ in one row, i.e.\ where 
\be \label{one row}
p_\eps(z,y) \simeq \tilde p_\eps(z,y) \text{ for some } z \in S, \text{ and } 
p_\eps(x,y) = \tilde p_\eps(x,y) \text{ for all other } x \in S.
\ee
Once this is done, we can exploit
the assumption that $S$ is finite, iteratively change row after row, 
and prove the full claim. 
For the case where \eqref{one row} holds, first note that for all $x \in S$, 
$$
\bbP^x_\eps (\tau_B < \tau_A, \tau_B \leq \tau_z) = \tilde \bbP^x_\eps 
(\tau_B < \tau_A, \tau_B \leq \tau_z). 
$$
This can be seen by considering a coupling $(X^{(\eps)}, \tilde X^{(\eps)})$ of the 
chains and by observing that by \eqref{one row}, 
$\bbP_{\rm{coupling}}^{(x,x)}(X_j^{(\eps)} = \tilde X_j^{(\eps)} \text{ for all } j \leq n,
\tau_{(z,z)} \geq n) = 1$ for all $n$.
So, the first time when the chains $X^{(\eps)}$ and $\tilde X^{(\eps)}$ can differ
is after they hit $z$. Thus, 
$$
\bbP^x_\eps(\tau_B < \tau_A)  = \tilde \bbP^x_\eps(\tau_B < \tau_A, \tau_B \leq \tau_z) + \bbP^x_\eps(\tau_B < \tau_A, \tau_z < \tau_B). 
$$
Since $\{\tau_z < \tau_B, \tau_z = \infty \} = \emptyset$, 
we can now use the strong Markov property to find 
$$
\bbP^x_\eps(\tau_B < \tau_A, \tau_z < \tau_B ) =  
\bbP^x_\eps(\tau_z < \tau_B) \bbP^z_\eps(\tau_B < \tau_A).
$$
Again $\bbP^x_\eps(\tau_z < \tau_B) = \tilde \bbP_\eps^x(\tau_z < \tau_B)$, and 
it remains to show that $\bbP^z_\eps(\tau_B < \tau_A) \simeq \tilde \bbP^z_\eps(\tau_B < \tau_A)$. 
If $z \in A \cup B$, this is trivial.
For $z \notin A \cup B$, 
$\bbP^z_\eps(\tau_B < \tau_A) = \bbP^z_\eps(\tau_B^+ < \tau_A^+)$, and 
$\tilde \bbP^z_\eps(\tau_B < \tau_A) = \tilde \bbP^z_\eps(\tau_B^+ < \tau_A^+)$. 
We are aiming 
to use Lemma \ref{committor formula}, and thus need to deal with the possibility that  
$\bbP^z_\eps(\tau_B^+ < \infty) = 0$. 

We assumed $p(x,y) \simeq \tilde p(x,y)$ for all $x,y \in S$, and so we also have 
$\bbP^x_\eps(\gamma) \simeq \tilde \bbP^x_\eps(\gamma)$ 
for each direct path from $x$ to $B$. By the definition of $\simeq$, a 
direct path $\gamma$ from $x$ to 
$B$ fulfills $\bbP^x_\eps(\gamma) > 0$ for all $\eps > 0$ 
in a neighborhood of $\eps=0$ if and only if 
$\tilde \bbP^x_\eps(\gamma) > 0$ in a neighborhood of $0$. Let us first assume that 
no such direct path exists. Then $\bbP_\eps^z(X_n^{(\eps)} \in B) = 
\tilde \bbP_\eps^z(X_n^{(\eps)} \in B) = 0$ for all $n \in \bbN$, and thus 
$\bbP^z(\tau_B < \tau_A) = \tilde \bbP_\eps^z(\tau_B<\tau_A) = 0$. Now let us assume
that such direct paths do exist.  
Since $\bbP^x_\eps(\tau_B < \infty) \geq \bbP^x_\eps(\gamma)$, we can use 
Lemma \ref{committor formula} to get 
\[
\bbP^z_\eps(\tau_B^+ < \tau_A^+) = \frac{\bbP^z_\eps(\tau_B^+ < \tau_{A \cup \{z\}}^+)}
{\bbP^z_\eps(\tau_B^+ < \tau_z^+)}.
\]
Now, 
\[
\begin{split}
& \bbP^z_\eps(\tau_B^+ < \tau_{A \cup \{z\}}^+)  = \sum_{w \in S} p_\eps(z,w) 
\bbP_\eps^w(\tau_B < \tau_{A \cup \{z\}}) \\
& \simeq\sum_{w \in S} \tilde p_\eps(z,w) 
\bbP_\eps^w(\tau_B < \tau_{A \cup \{z\}})
 = \sum_{w \in S} \tilde p_\eps(z,w) 
\tilde \bbP_\eps^w(\tau_B < \tau_{A \cup \{z\}}) = \tilde \bbP_\eps^z(\tau_B^+ < \tau_{A \cup \{z\}}^+),
\end{split}
\]
and the same argument shows $\bbP_\eps^z(\tau_B^+ < \tau_z^+) \simeq 
\tilde \bbP_\eps^z(\tau_B^+ < \tau_z^+)$. The claim follows. 
\end{proof}

The statement of Theorem \ref{chain comparison} is rather surprising. 
The reason is that even though in each step 
that the chain takes from $x$ on its way to $B$, the probabilities for the chains
$X$ and $\tilde X$ differ only by a factor that becomes negligibly close to one
as $\eps \to 0$, in the same limit the number of steps needed to reach $B$ can 
diverge. Indeed, imagine two $P_0$-essential classes $E$ and $E'$ that are linked by  
direct paths $\gamma$ 
with $\bbP_\eps(\gamma) = \caO(\eps)$, but are linked to $A$ and $B$ only by paths
$\gamma'$ with $\bbP_\eps(\gamma') = \caO(\eps^2)$. Then, starting from a point in $E$, both
$E$ and $E'$ will be visited
many times before either $A$ or $B$ is hit. So one could fear that the 
errors committed by changing each transition probability to an asymptotically 
equivalent one will pile up; but as Theorem \ref{chain comparison} shows,
this is not the case.

\begin{theorem} \label{magic formula}
Let $E$ be a $P_0$-essential class, $x \in E$ and $z \notin E$. Then 
\be \label{magic formula 1}
\bbP^x_\eps(\tau_{E^c}^+ < \tau_x^+, X_{\tau_{E^c}} = z) \simeq 
\frac{1}{\nu_E(x)} \sum_{y \in E} \nu_E(y) p_\eps(y,z),
\ee
and 
\be \label{magic formula 2}
\bbP^x_\eps(X_{\tau_{E^c}} = z) \simeq \frac{1}{Z_\eps(E)}
\sum_{y \in E} \nu_E(y) p_\eps(y,z),
\ee
with normalizing constant 
$$
Z_\eps(E) = \sum_{\tilde z \in E^c} 
\sum_{\tilde y \in E} \nu_E(\tilde y) p_\eps(\tilde y, \tilde z).
$$
\end{theorem}

{\bf Remark:} \eqref{magic formula 2} is intuitively clear: for small $\eps$, 
the Markov chain spends such a long time in $E$ before exiting that it essentially
exits $E$ from its $E$-stationary distribution. Formula \eqref{magic formula 1} on 
the other hand is rather remarkable, since a return to $x$ happens in a 
time of order one, so there is no time for the chain to become stationary.

\begin{proof}[Proof of Theorem \ref{magic formula}]
To prove \eqref{magic formula 1}, choose $A = E^c \cup \{x\}$ in Proposition 
\ref{first formula}. Then 
$$
\bbP^x_\eps(\tau_{E^c}^+ < \tau_x^+, X_{\tau_{E^c}} = z) = 
\bbP_\eps^x(X_{\tau_A^+} = z) = p_\eps(x,z) + \sum_{y \in E \setminus \{x\}}
\frac{\bbP_\eps^x(\tau_y^+ < \tau_{E^c \cup \{x\}}^+)}
{\bbP_\eps^y(\tau_{E^c \cup \{x\}}^+ < \tau_y^+)} p_\eps(y,z).
$$
We decompose
\be \label{simple decomp}
\bbP_\eps^x(\tau_y^+ < \tau_{E^c \cup \{x\}}^+) = 
\bbP_\eps^x(\tau_y^+ < \tau_{E^c \cup \{x\}}^+, \tau_{E^c}^+ > \tau_x^+) + 
\bbP_\eps^x(\tau_y^+ < \tau_{E^c \cup \{x\}}^+,  \tau_{E^c}^+ < \tau_x^+).
\ee
The first term is equal to 
$\bbP_\eps^x(\tau_y^+ < \tau_x^+) - 
\bbP_\eps^x(\tau_y^+ < \tau_x^+, \tau_{E^c}^+ < \tau_x^+)$. The second term in 
this decomposition as well as the second term in \eqref{simple decomp} are bounded by $\bbP_\eps^x(\tau_{E^c}^+ < \tau_x^+)$ and thus vanish $\eps \to 0$,
due to Lemma \ref{like unperturbed} and the finiteness of $E^c$. The same Lemma then yields  
$$
\lim_{\eps \to 0} \bbP_\eps^x(\tau_y^+ < \tau_{E^c \cup \{x\}}^+) = \bbP_0^x(\tau_y^+ < \tau_x^+).
$$
Similarly, for $y \in E$ we have 
$$
\lim_{\eps \to 0} \bbP_\eps^y(\tau_{E^c \cup \{x\}}^+ < \tau_y^+) = 
\lim_{\eps \to 0} \bbP_\eps^y(\tau_x^+ < \tau_y^+) = \bbP_0^y(\tau_x^+ < \tau_y^+).
$$
By Proposition \ref{basic fact}, we conclude 
$$ 
\lim_{\eps \to 0} \frac{\bbP_\eps^x(\tau_y^+ < \tau_{E^c \cup \{x\}}^+)}
{\bbP_\eps^y(\tau_{E^c \cup \{x\}}^+ < \tau_y^+)} = 
\frac{\bbP_0^x(\tau_y^+ < \tau_x^+)}{\bbP_0^y(\tau_x^+ < \tau_y^+)} = 
\frac{\nu_E(y)}{\nu_E(x)}.
$$
Here, we have used that $\bbP_0^y(\tau_x^+ < \tau_y^+) > 0$ for all $x,y \in E$. Since 
the right-hand side above is strictly positive, we conclude that 
$\bbP_\eps^x(\tau_y^+ < \tau_{E^c \cup \{x\}}^+) / 
\bbP_\eps^y(\tau_{E^c \cup \{x\}}^+ < \tau_y^+) \simeq \nu_E(y) / \nu_E(x)$,
and \eqref{magic formula 1} is shown. 

To see \eqref{magic formula 2}, we use Proposition \ref{first formula} with 
$A = E^c$. Since now $x \notin A$, we can use \eqref{hitting prob 1b} and obtain
\be \label{easy}
\bbP^x_\eps(X_{\tau_{E^c}^+} = z) = \sum_{y \in E} 
\frac{\bbP^x(\tau_y < \tau_{E^c}^+)}{\bbP^y_\eps(\tau_{E^c}^+ < \tau_y^+)} p_\eps(y,z).
\ee
As before, $\bbP^x(\tau_y < \tau_{E^c}^+) \simeq 1$ for all $y \in E$. 
By summing \eqref{magic formula 1} over all $\tilde z \notin E$, we get 
\[
\bbP_\eps^y (\tau_{E^c}^+ < \tau_y^+) \simeq \frac{1}{\nu_E(y)} 
\sum_{\tilde y \in E, \tilde z \notin E} \nu_E(\tilde y) 
p_\eps(\tilde y, \tilde z).
\]
Plugging these into \eqref{easy}, we obtain \eqref{magic formula 2}.
\end{proof}

\section{Perturbed Markov chains: metastable dynamics}
\label{dynamics}

Here we describe the metastable dynamics of a perturbed Markov chain. As in the 
previous section, we will restrict our attention to a finite state space $S$ throughout.

First of all, we have to define what we mean by metastable dynamics. 
We follow the theory of Bovier et al \cite{BEGK1,BEGK2,BGK}. In the case of perturbed
Markov chains on a finite state space, 
Definition 2.1 from \cite{BEGK1} (see also \cite{BovNotes}) goes as follows:
a set $M \subset S$ is called a {\em set of metastable points} if for all $x \in M$ 
and $y \notin M$, 
\be \label{Bovier def}
\lim_{\eps \to 0} \frac{\bbP_\eps(\tau_{M\setminus \{x\}}^+ < \tau_x^+)}
{\bbP_\eps^y(\tau_M^+ < \tau_y^+)} = 0.
\ee
In words, this means that reaching $M$ from the outside of $M$ is much easier than 
traveling between different points of $M$, in both cases with the restriction not to return to one's starting point first. 

Using Lemma \ref{direct path lemma} and Lemma \ref{like unperturbed}, it is easy to see that
if we choose precisely one point from each of the $P_0$-essential classes 
$E_1, \ldots, E_n$, then the set $S_0 = \{x_1, \ldots, x_n\}$ is a set of metastable points.
Also, $S_0$ is maximal in the sense that adding a further point to $S_0$ will result in 
a set no longer fulfilling \eqref{Bovier def}. On the other hand, removing points from 
$S_0$ or replacing them with points from $F$ may in certain cases still result in a 
metastable set, depending on the structure of the Markov chain and the points in question. 
We will not pursue this further since $S_0$ is the most natural choice. Of course, 
when some of the $E_i$ contain more than one point, the choice of $S_0$ is not unique. 
One of our main results is that when defining the effective chain by the transition matrix 
\be \label{p hat}
\begin{split}
\hat p_\eps(x_i,x_j) & := \nu_{E_i}(x_i) \bbP^{x_i}_\eps(X_{\tau_{S_0}^+} = x_j) \quad 
\text{ for } i \neq j, \\
\hat p_\eps(x_i,x_i) & := \nu_{E_i}(x_i) \bbP^{x_i}_\eps(X_{\tau_{S_0}^+} = x_i) + 1 - \nu_{E_i}(x_i),
\end{split}
\ee
then the relevant dynamical quantities will be asymptotically 
independent of the choice of the 
representatives $x_i$. 

The occurrence of the expression $\bbP^{x_i}_\eps(X_{\tau_{S_0}^+} = x_j)$ 
in \eqref{p hat} is intuitively obvious, since 
it means that we just monitor the chain when it hits one of our reference points $x_j$. 
The factor $\nu_{E_i}(x_i)$ may be less obvious.
To motivate it, note that  
by \eqref{magic formula 1},
\be \label{explanation}
\begin{split}
\bbP^{x_i}_\eps(X_{\tau_{S_0}^+} = x_j) & = \sum_{z \in S \setminus E_i} 
\bbP_\eps^x(\tau_{E^c} < \tau_{x_i}^+, X_{\tau_{E^c}} = z) 
\bbP^z_\eps(X_{\tau_{S_0}} = x_j) \\ 
& \simeq \frac{1}{\nu_{E_i}(x_i)}
\sum_{w \in E_i, z \notin E_i} \nu_{E_i}(w) p_\eps(w,z) \bbP^z_\eps(X_{\tau_{S_0}} = x_j).
\end{split}
\ee
This shows that the factor 
$\nu_{E_i}(x_i)$ in \eqref{p hat} cancels one of the dependencies of 
$\bbP^{x_i}_\eps(X_{\tau_{S_0}^+} = x_j)$ on 
the choice of our set $S_0$. While the terms $\bbP^z_\eps(X_{\tau_{S_0}^+} = x_j)$ still
do depend on the choice of $S_0$, we will see below 
that including the factor $\nu_{E_i}(x_i)$
in the definition is enough to obtain the asymptotically correct stationary distribution
and escape probabilities. This justifies the following definition:

\begin{definition}
Let $X^{(\eps)}$ be an irreducibly perturbed Markov chain on a finite state space. 
The Markov chain $\hat X^{(\eps)}$ with state space $S_0$ and transition matrix  
\eqref{p hat} is called 
the {\bf effective metastable representation} of $X^{(\eps)}$ corresponding to $S_0$. 
\end{definition}

In order to show the properties of the chain $\hat X^{(\eps)}$ announced above, 
we define a second effective Markov chain, this time
without reference to a set of representatives. 
For $E,E' \in \caE$ with $E \neq E'$ we put 
\be \label{q hat}
\hat q_\eps(E,E') := \sum_{x \in E} \nu_E(x)^2 \,\,\bbP_\eps^x( \tau_{E'}^+ < \tau_x^+),
\ee
and $\hat q_\eps(E,E) := 1 - \sum_{E' \in \caE \setminus \{E\}} \hat q_\eps(E,E')$. 
The $\hat q_\eps$ are the elements of a transition matrix when $\eps$ is sufficiently small. 
As the following Proposition shows, 
this chain is reversible and the reversible measure of 
$E \in \caE$ is $\mu(E)$:

\begin{proposition} \label{asymptotic detailed balance}
The quantities $\hat q_\eps$ satisfy the asymptotic detailed balance equation
$$
\mu_\eps (E) \hat q_\eps(E,E') \simeq \mu_\eps(E') \hat q_\eps(E',E).
$$
\end{proposition}

The proof of Proposition \ref{asymptotic detailed balance} rests on the following simple lemma:

\begin{lemma} \label{E instead of x}
Let $E,E'$ be $P_0$-essential classes, $E \neq E'$, $x \in E$, $y \in E'$, and $z \in S$. Then
\be \label{all the same}
\bbP_\eps^z(\tau_y^+ < \tau_x^+) \simeq \bbP_\eps^z(\tau_{E'}^+ < \tau_x^+).
\ee
\end{lemma}

\begin{proof} 
From Lemma \ref{like unperturbed}, we have 
$\bbP^{\tilde y}(\tau_y < \tau_x) \simeq 1$ 
for all $\tilde y \in E'$. Since $\{\tau_y^+ < \tau_x^+\} \subset 
\{\tau_{E'}^+ < \tau_x^+\}$ for all $x \in E$, the strong Markov property gives
\[
\begin{split}
\bbP^z_\eps(\tau_y^+ < \tau_x^+) & = \sum_{\tilde y \in E'} 
\bbP^z_\eps(\tau_{E'}^+ < \tau_x^+, X_{\tau_{E'}^+} = \tilde y) 
\bbP_\eps^{\tilde y}(\tau_y < \tau_x) \\
& \simeq \sum_{\tilde y \in E'} 
\bbP^{z}_\eps(\tau_{E'}^+ < \tau_x^+, X_{\tau_{E'}^+} = \tilde y) 
= \bbP^z_\eps(\tau_{E'}^+ < \tau_x^+) 
\end{split}
\]
\end{proof}

\begin{proof}[Proof of Proposition \ref{asymptotic detailed balance}]
When $E=E'$, the claim holds trivially. For $E \neq E'$, pick $x \in E$ and $y \in E'$.
We use Corollary \ref{structure} b), Proposition \ref{basic fact}, and 
Corollary \ref{structure} b) again to find 
\[
\begin{split}
\mu_\eps(E) \nu_E(x) \bbP_\eps^x(\tau_y^+ < \tau_x^+) & \simeq 
\mu_\eps(x) \bbP_\eps^x(\tau_y^+ < \tau_x^+) 
= \mu_\eps(y) \bbP_\eps^y(\tau_x^+ < \tau_y^+) \\
&  \simeq \mu_\eps(E') \nu_{E'}(y) \bbP_\eps^y (\tau_x^+ < \tau_y^+).
\end{split}
\]
Thus by \eqref{all the same}, 
\be \label{no x,y dependence}
\mu_\eps(E) \nu_E(x) \bbP_\eps^x(\tau_{E'}^+ < \tau_x^+) \simeq 
\mu_\eps(E') \nu_{E'}(y) \bbP_\eps^y(\tau_{E}^+ < \tau_y^+),
\ee
for all $x \in E$. Since the right-hand side is independent of $x$, and the 
left-hand side is independent of $y$, we find 
\be \label{really no x dependence}
\nu_E(x) \bbP_\eps^x(\tau_{E'}^+ < \tau_x^+) \simeq \nu_E(\tilde x) \bbP_\eps^{\tilde x} (\tau_{E'}^+ < \tau_{\tilde x}^+)
\ee
for all $x,\tilde x \in E$, and similarly for $E'$. Thus when we multiply 
\eqref{no x,y dependence} with $\nu_E(x) \nu_{E'}(y)$ and sum over $x \in E$ and 
$y \in E'$, we obtain the claim. 
\end{proof}

The next result shows that the effective metastable representation $\hat X^{(\eps)}$ 
indeed describes the metastable dynamics of $X^{(\eps)}$ correctly, 
in the sense that asymptotically it 
has the right escape probabilities and thus the right stationary distribution. 
Let us write $\hat \mu_\eps$ for 
the stationary distribution and $\hat \bbP_\eps$ for the path measure of 
$\hat X^{(\eps)}$. 

\begin{theorem} \label{escape probs}
For $i \neq j$, we have $
\hat \bbP_\eps^{x_i}(\tau_{x_j}^+ < \tau_{x_i}^+) \simeq \hat q_\eps(E_i,E_j)$.
In particular $\hat \mu_\eps(x_i) \simeq \mu_\eps(E_i)$. 
\end{theorem}

\begin{proof}
From \eqref{really no x dependence} and Lemma \ref{E instead of x},
we see that 
$$
\hat q_\eps(E_i,E_j) \simeq \nu_{E_i}(x_i) \bbP_\eps^{x_i}(\tau_{E_j}^+ < \tau_{x_i}^+) 
\simeq \nu_{E_i}(x_i) \bbP_\eps^{x_i}(\tau_{x_j}^+ < \tau_{x_i}^+). 
$$
The Markov property and the definition of $\hat P_\eps$ then gives
$$
\hat q_\eps(E_i,E_j) \simeq \hat p_\eps(x_i,x_j) + 
\sum_{k \neq i,j} \hat p_\eps(x_i,x_k) \bbP_\eps^{x_k}(\tau_{x_j}^+ < \tau_{x_i}^+).
$$
We will show below that for $k \neq i,j$,
\be \label{remains}
\bbP_\eps^{x_k}(\tau_{x_j}^+ < \tau_{x_i}^+) = 
\hat \bbP_\eps^{x_k}(\tau_{x_j}^+ < \tau_{x_i}^+).
\ee
Once this is done, the Markov property for $\hat \bbP_\eps$ shows the first claim, 
and from Proposition \ref{asymptotic detailed balance} we get
$$
\mu_\eps(E_i) \hat \bbP^{x_i}(\tau_{x_j} < \tau_{x_i}) \simeq 
\mu_\eps(E_i) \hat q(E_i,E_j) \simeq \mu_\eps(E_j) \hat q(E_j,E_i) \simeq 
\mu_\eps(E_j) \hat \bbP^{x_j}(\tau_{x_j} < \tau_{x_i}).
$$
Since 
$\hat \mu_\eps(x_i) \hat \bbP^{x_i}(\tau_{x_j} < \tau_{x_i}) 
= \hat \mu_\eps(x_j)  \hat \bbP^{x_j}(\tau_{x_i} < \tau_{x_j})$ by Proposition
\ref{basic fact}, we get 
\[
\frac{\hat \mu_\eps(x_i)}{\hat \mu_\eps(x_j)} \simeq 
\frac{\mu_\eps(E_i)}{\mu_\eps(E_j)}.
\]
Since $\sum_j \mu_\eps(E_j) \simeq 1$ by Corollary \ref{structure} a), we can sum 
over $i$ and obtain $1 / \hat \mu_\eps(x_i) \simeq 1 / \mu_\eps(E_i)$, and thus 
$\hat \mu_\eps(x_i) \simeq \mu_\eps(E_i)$.

To show \eqref{remains}, we introduce the shorthand 
$$\nu_k = \nu_{E_k}(x_k), \quad  p(k,l) = \bbP^{x_k}(X_{\tau_{S_0}^+} = x_l), \quad 
\hat p(k,l) = \hat p_\eps(x_k,x_l) = \hat \bbP^{x_k}(\hat X_{\tau_{S_0}^+} = x_l).
$$
From \eqref{p hat}, we get $\hat p(k,l) = \nu_k p(k,l) + (1 - \nu_k) \delta_{k,l}$.
Now a standard application of the Markov property with the stopping time 
$\tau_{S_0}^+$ shows that for $k \neq i,j$, $k \mapsto h(k) = \bbP_\eps^{x_k}(\tau_{x_j}^+ < 
\tau_{x_i}^+)$ is the unique solution of the harmonic equation 
$\sum_{l=1}^n p(k,l) h(l) = h(k)$ for all $k \neq i,j$ with boundary conditions 
$h(i) = 0$, $h(j)=1$. 
Likewise,  $k \mapsto \hat h(k) = \hat \bbP_\eps^{x_k}(\tau_{x_j}^+ < 
\tau_{x_i}^+)$ is the  unique solution of the harmonic equation 
$\sum_{l=1}^n \hat p(k,l) \hat h(l) = \hat h(k)$ for all $k \neq i,j$ 
with boundary conditions $\hat h(i) = 0$, $\hat h(j)=1$. But since 
$$
\sum_{l=1}^n \hat p(k,l) h(l) = \nu_k \sum_{l=1}^n p(k,l) h(l) + (1 - \nu_k) h(k) = 
\nu_k h(k) + (1 - \nu_k) h(k) = h(k),
$$
we must have $\hat h(k) = h(k)$, and the claim follows. 
\end{proof}

The advantage of the chain $\hat X^{(\eps)}$ is that its transition matrix 
is almost diagonal in the sense that 
$\lim_{\eps \to 0} \hat p_\eps(x_i,x_j) = \delta_{i,j}$. In particular, $\hat X^{(\eps)}$
is an irreducible perturbation of the trivial (identity) Markov chain. It is now 
natural to rescale time so that the most likely transition between two different
states becomes of order one. More precisely, we set 
\be \label{p check}
\check p_\eps(x_i,x_j) := \frac{\hat p_\eps(x_i,x_j)}{\sum_{k,l: k \neq l} \hat p_\eps(x_k,x_l)}, \qquad \check p_\eps(x_i,x_i) := 1 - \sum_{j:j \neq i} \check p_\eps(x_i,x_j).
\ee
Since $\sum_{k,l:k \neq l} \check p_\eps(x_k,x_l) = 1$, for each $\eps>0$ at least one 
of the terms in the finite sum must be large. The problem is that at this point
we cannot guarantee that the quantities $\check p_\eps(x_i,x_j)$ converge. To see 
what could happen, consider the example $S = \{x,y\}$, 
$p_\eps(x,y) = \eps (2 + \sin(1/\eps))$, $p_\eps(y,x) = \eps$. Then $\hat P = P$,
but $\check p_\eps(y,x) = \frac{1}{3 + \sin(1/\eps)}$ does not converge. Of course,
this also implies that $\lim_{\eps \to 0} \mu_\eps$ does not exist. 

So far, we did not have to pay attention to that type of problem - all of our results
above are valid as asymptotic equivalences, whether or not the quantities in question 
converge. Now however, we need proper convergence to carry on, and will give a 
sufficient criterion. Let $\eps \mapsto a_\eps$, 
$\eps \mapsto b_\eps$ be two functions of $\eps > 0$. We say that $a_\eps$ and 
$b_\eps$ are {\em asymptotically comparable}, and write $a_\eps \sim b_\eps$, 
if either both of them are 
strictly positive and $\lim_{\eps\to 0} a_\eps / b_\eps$ exists in $[0,\infty]$, or 
if one or both of them are identically zero. Note that we allow $0$ and $\infty$ as 
possible limits. We caution the reader that unlike asymptotic equivalence, asymptotic
comparability is not transitive, and is not stable under multiplications. On the other hand,
it is obviously symmetric, and we have the following summability property: 
If $a_\eps, b_\eps$, and $c_\eps$ are mutually asymptotically comparable, and if
$\alpha_\eps, \beta_\eps$ and $\gamma_\eps$ have strictly positive, finite limits as 
$\eps \to 0$, then 
\be \label{asym comp summ}
\alpha_\eps a_\eps + \beta_\eps b_\eps \sim \gamma_\eps c_\eps.
\ee 
We say that an irreducibly perturbed Markov chain $X^{(\eps)}$ is {\em regular} if for 
all $m,n \in \bbN$ and all sequences of pairs $(x_i,y_i)_{i \leq n}$, 
$(z_i,w_i)_{i \leq m}$ with $x_i,y_i,z_i,w_i \in S$, we have 
\be \label{regular def}
\prod_{i=1}^n p_\eps(x_i,y_i) \sim \prod_{i=1}^m p_\eps(z_i,w_i).
\ee
We will call a transition matrix $P$ regular if the generated Markov chain is regular.

Examples of regular perturbed Markov chains include those treated in \cite{WiGr2},
where the transition elements 
are of the form $c_\eps(x,y) \eps^{k(x,y)}$ with $c_\eps$ either converging to a 
strictly positive limit or identically zero, and $k(x,y)$ independent of $\eps$. They also
include those with property $\caP$ introduced in \cite{OlSc2}.

\begin{theorem} \label{conservation of regularity}
For a regular perturbed Markov chain with transition matrix $P_\eps$, define 
$\hat P_\eps$ as in \eqref{p hat}, and $\check P_\eps$ as in \eqref{p check}.
Then $\hat P_\eps$ and $\check P_\eps$ are transition matrices of 
regular perturbed Markov chains. 
\end{theorem}

\begin{proof}
By \eqref{explanation}, for $i \neq j$ 
\be \label{aa0}
\hat p_\eps(x_i,x_j) \simeq \sum_{w \in E_i, z \notin E_i} \nu_{E_i}(w) p_\eps(w,z) \bbP^z_\eps(X_{\tau_{S_0}} = x_j), 
\ee
and Proposition \ref{second formula} gives 
\be \label{aaa}
\bbP^z_\eps (X_{\tau_{S_0}} = x_j) = \sum_{\gamma \in \Gamma_{S_0^c}(z,x_j)} 
\prod_{i=1}^{|\gamma|-1} \frac{p_\eps(\gamma_i,\gamma_{i+1})}
{1 - \bbP_\eps^{\gamma_i}(X_{\tau^+_{S_0 \cup \{\gamma_1, \ldots, \gamma_i\} }} = \gamma_i)}
\ee
In Lemma \ref{regularity lemma} below we will show that if $S_0$ contains one representative
of each $P_0$-essential class then  
$\lim_{\eps \to 0} \bbP_\eps^{\gamma_i}(X_{\tau^+_{S_0 \cup \{\gamma_1, \ldots, \gamma_i
\} }} = \gamma_i)$ exists and is strictly smaller than one for all $\gamma$. Thus each
$\lim_{\eps\to 0} 1 / (1-  \bbP_\eps^{\gamma_i}(X_{\tau^+_{S_0 \cup \{\gamma_1, \ldots, 
\gamma_i\} }} = \gamma_i)) \geq 1$ exists. In other words, 
$\bbP^z_\eps (X_{\tau_{S_0}} = x_j)$
is given as a sum of terms of the form $c_\eps(z_1, \ldots, z_{n+1}) \prod_{i=1}^n p_\eps(z_i,z_{i+1})$ with 
$z_i \in S_0^c \cup \{z,x_j\}$, where 
$\lim_{\eps \to 0} c_\eps(z_1, \ldots, z_{n_1}) \geq 1$ exists for all 
$(z_1, \ldots, z_n)$. 
When plugging this into \eqref{aa0}, we can apply the extension of 
\eqref{asym comp summ} to finite sums to show that $\hat P_\eps$ is the transition
matrix of a regular Markov chain. 
By \eqref{p check}, this immediately implies convergence of the transition
probabilities $\check p_\eps(x_i,x_j)$. 
Rewriting the second equation in \eqref{p check} in the form 
\[
\check p_\eps(x_i,x_i) = \tfrac{\sum_{k,l:k \notin \{l,i\} } \hat p_\eps(x_k,x_l)}
{\sum_{k,l: k \neq l} \hat p_\eps(x_k,x_l)},
\]
we see in addition that the chain $\hat X^{(\eps)}$ is a regular perturbed Markov chain. 
\end{proof}

It remains to prove the claim used in the proof above.

\begin{lemma} \label{regularity lemma}
Let $X^{(\eps)}$ be a perturbed Markov chain. 
Assume that a set $S_0$ contains one element of each $P_0$-essential class. 
Let $A \subset S$ with $S_0 \subset A$. Then for all $x \in A \setminus S_0$, 
$\lim_{\eps \to 0} \bbP_\eps^x(X_{\tau_A^+} = x)$ exists and is strictly smaller than
$1$.
\end{lemma}

\begin{proof}
As $S_0$ contains a representative of each $P_0$-essential class, 
there must be a $P_0$-relevant direct path $\gamma$ from $x$ to some 
$y \in S_0$. So, 
$\limsup_{\eps \to 0} \bbP_\eps^x(X_{\tau_A^+} = x) < 1 - 
\lim_{\eps \to 0} P_\eps(\gamma) < 1$. 

For the existence of the limit, let first $A := S$. For $x \notin S_0$, 
we have $\bbP_\eps(X_{\tau_A^+} = x) = p_\eps(x,x) \to p_0(x,x)$ as 
$\eps \to 0$. 
Let us now assume that the claim holds for all $\bar A$ such that 
$|\bar A| \geq |S| - k + 1$ with some $k \in \bbN$. Let $A$ be such that 
$|A| = |S|-k$. Then, 
\[
\begin{split}
\bbP_\eps^x(X_{\tau_A^+}=x) & = p_\eps(x,x) + 
\sum_{y \notin A} p_\eps(x,y) \bbP_\eps^y(X_{\tau_A} = x) \\
& = 
p_\eps(x,x) + \sum_{y \notin A} p_\eps(x,y) \sum_{\gamma \in \Gamma_{S \setminus A}(y,x)}
\prod_{i=1}^{|\gamma|-1} \frac{p_\eps(\gamma_i,\gamma_{i+1})}
{1 - \bbP_\eps^{\gamma_i}(X_{\tau_{A \cup \{\gamma_1, \ldots, \gamma_i\}}} = \gamma_i)}.
\end{split}
\]
By the induction hypothesis,
$\lim_{\eps \to 0} \bbP_\eps^{\gamma_i}(X_{\tau_{A \cup \{\gamma_1, \ldots, \gamma_i\}}} = \gamma_i)$ exists and is strictly smaller than $1$. Thus also 
$\lim_{\eps \to 0} \bbP_\eps^x(X_{\tau_A^+}=x)$ exists and is strictly smaller than $1$. 
The claim follows by induction. 
\end{proof}

We have thus found a way to successively describe the multi-scale metastable dynamics 
of regular perturbed Markov chains: starting with the original chain $X^{(\eps)}$, 
we derive $\hat X^{(\eps)}$ and then $\check X^{(\eps)}$. By Theorem 
\ref{conservation of regularity}, $\hat X^{(\eps)}$ and $\check X^{(\eps)}$ are 
again regular perturbed 
Markov chains. Moreover, all of the $\hat P_0$-essential classes consist of 
exactly one element, and $\lim_{\eps \to 0} \hat p_\eps(x_i,x_j) = 0$ whenever 
$i \neq j$. So, $\hat P_\eps$ describes the effective metastable 
dynamics, but still in the original time scale. 

The transformation from $\hat P_\eps$ to $\check P_\eps$ means that we 
go to a time scale where the most likely transitions between different
states become of order one. In other words, there exist $i \neq j$ 
with $\lim_{\eps \to 0} \check p(x_i,x_j) > 0$.  
By Lemma \ref{direct path lemma}, this implies that $\{x_i\}$ will 
no longer be a $\check P_0$-essential class on its own: it will either form a larger $\check P_0$-essential class together with some $\{x_j\}$, $j \neq i$, or it will have become
$\check P_0$-transient. In any case, the number of $\check P_0$-essential classes will
be smaller than the number of $P_0$-essential classes. Thus by 
applying the transformations $P_\eps \to \hat P_\eps
\to \check P_\eps$ to the matrix $\check P_\eps$, and iterating the procedure, 
we can recursively explore longer and longer time scales of the dynamics. 

On a purely theoretical level, our theory of multi-scale metastable dynamics for 
regular perturbed Markov chains is thus complete. However, if one attempts to 
(numerically) compute the transition probabilities at the different time scales, 
the problem arises that all relevant expressions in our theory still contain terms 
of the form $\bbP^z_\eps(X_{\tau_{S_0}} = x_j)$. In the next section, we will show why 
naive attempts to compute this quantity numerically are likely to fail, and present
a numerically stable algorithm for computing them. A byproduct of our algorithm is
a numerically stable method to compute the matrix elements of the 
transition matrix $\hat Q_\eps$, and thus the stationary weights $\mu_\eps(E)$ for 
all $P_0$-essential classes $E$. 

\section{Computing hitting probabilities and the asymptotic stationary distribution}
\label{algorithms}

This section deals with aspects of the numerical computation of the transition 
probabilities $\hat p_\eps$ and $\hat q_\eps$ given in \eqref{p hat} and 
\eqref{q hat}, respectively. Before we proceed we would like to make clear that subtle 
issues coming from the field of \emph{computable analysis} fall beyond the scope of this article. 
Intuitively though, we mean the following by "numerical computation": if someone enumerates, step by 
step, all members of an infinite sequence of transition matrices $P_{\eps_n}$ that converge towards $P_{0}$, 
we are able to process each $P_{\eps_n}$ using only a computer and produce, step by step, an infinite sequence 
that converges towards the matrices $\hat p_0$ (or $\hat q_0$). This corresponds roughly to the property of 
being \emph{computably approximable}. Note that in this case we do not know how fast the sequence is 
converging to the limit. Said otherwise, if we want a precise approximation of, say, $\hat p_0$, we have no idea until 
which $\eps_n$ we should process the $P_{\eps_n}$.  
This is a usual issue in numerical analysis. If in addition we would know that the $n$-th approximation is, \textit{e.g.}, 
at most $2^{-n}$ away from the limit, we would know when to stop to obtain the desired precision. 
This corresponds roughly to the property of being \emph{computable}.

The starting point of our considerations are the formulae 
\be \label{p eps numeric} 
\hat p_\eps(x_i,x_j) = 
\nu_{E_i}(x_i) \Big( \sum_{j \neq i} p_\eps(x_i,x_j) + \sum_{z \notin S_0} 
p_\eps(x_i,z) \bbP_\eps^{z} (X_{\tau_{S_0}} = x_j) \Big),
\ee
and 
\be \label{q eps numeric}
\hat q_\eps(E,E') = \sum_{x \in E} \nu_E(x)^2 \Big( \sum_{y \in E'} p_\eps(x,y) + 
\sum_{z \notin \{x\} \cup E'} p_\eps(x,z) \bbP_\eps^z(\tau_{E'}^+ < \tau_x^+) \Big),
\ee
both of which are obtained from the definition of the respective quantities using the strong
Markov property. 
In both cases, the task is to compute a hitting probability of the form
\be \label{committor def}
h_{A,B}(z) := \bbP_\eps^z(X_{\tau_{A \cup B}} \in A),
\ee
where $A,B \subset S$ and $z \in S \setminus (A \cup B)$. 
In the case of $\hat q_\eps(E,E')$, $A = E'$ and $B = \{x\}$. 
Such hitting probabilities are well understood in the 
theory of Markov processes: $h_{A,B}$ is called the 
committor function in \cite{vdE1,vdE2} and the equilibrium potential of the 
capacitator $(B,A)$ in \cite{BovNotes}, and  
is the unique harmonic continuation from $C := A \cup B$ to $S$ of the 
indicator function $1_A$ of $A$. This means that $h_{A,B}$ is the unique solution of the linear system
\be \label{possibly ill conditioned} 
\sum_{z \in S \setminus C}(P_\eps(x,z) - \delta_{x,z}) h_{A,B}(z) = - r(x), \qquad x \in C^c = S  \setminus C,
\ee
where $r(x) := P_\eps 1_A(x)$. 
Let us write $\bar P_\eps = (p_\eps(x,y))_{x,y \in C^c}$ for the restriction of $P_\eps$
to $C^c$. 
If $C \neq \emptyset$ and $P_\eps$ is irreducible, we have seen in the proof of Proposition 
\ref{quotient} that $1-\bar P_\eps$ is invertible. 
We thus find the committor function by matrix inversion:
\be \label{matrix formulation}
h_{A,B}(x) = [(1- \bar P_\eps)^{-1} r] (x), \qquad x \in C^c.
\ee

The problem with this formula is that 
as $\eps \to 0$, the matrix $(1 - \bar P_\eps)$ may converge to
a non-invertible matrix. In that case, some matrix elements of 
$(1- \bar P_\eps)^{-1}$ will diverge, and even though the quantities $h_{A,B}(z)$ 
themselves are bounded by $1$ for all $\eps$, computing them numerically 
becomes unreliable as $\eps \to 0$. Our first result will identify situations where 
this cannot happen. 

We call a state $y \in S$ an {\em asymptotic dynamical trap} 
(or simply a trap) with respect to $C$ if 
$$
\liminf_{\eps \to 0} \bbP_\eps^y(\tau_C < M) = 0 \qquad \text{ for all }
M \in \bbN.
$$
A necessary condition for $y$ to be a trap is that there exists no direct 
$P_0$-relevant path from $y$ into $C$. On the other hand, for all $z \in S$ there  
is at least one $P_0$-relevant path from $y$ to some $P_0$-essential class. 
Thus if $C$ intersects all $P_0$-essential classes, no traps will exist. 
 
Recall that for a 
matrix $P$, the {\em condition number} is given by $\kappa(P) = \| P \| \| P^{-1} \|$, 
where $\|. \|$ is the operator norm with respect to any norm on the underlying vector space. 
In our case, it is convenient to use the supremum norm on the vector space. 

\begin{proposition} \label{condition nr}
Assume that $C \subset S$ is such that there are no asymptotic dynamical traps 
with respect to $C$. Then $\limsup_{\eps \to 0} \kappa(1 - \bar P_{\eps}) < \infty$. 
\end{proposition}

\begin{proof}
Since $\bar P_\eps$ is substochastic, clearly $\| \bar P_\eps \| \leq 1$. 
On the other hand, the absence of traps with respect to $C$ allows us to find 
$k_0 \in \bbN$ and $c < 1$, both of them independent of $\eps$, and so that 
$\bbP_\eps^x(\tau_C > k_0) \leq c$ for all $x \in S$. The strong Markov property then
implies $\bbP_\eps^x(\tau_C > k) \leq c^{\lfloor k/k_0 \rfloor}$ for all $k \in \bbN$ and 
all $x \in S$. Thus for $x \in C^c$ and bounded $f: C^c \to \bbR$, we find 
$$
\big| (\bar P_\eps)^k f (x) \big| = \big| \bbE^x_\eps(f(X_k) 1_{\{\tau_C > k\}}) \big| \leq 
\| f \|_{\infty} \bbP^x_\eps(\tau_C > k) \leq \| f \|_{\infty} 
c^{\lfloor k/k_0 \rfloor}.
$$
Consequently, the left-hand side above is absolutely summable, and 
$$
\big| (1 - \bar P_\eps)^{-1} f (x) \big| = \big| \sum_{k=0}^\infty 
( \bar P_\eps )^k f (x) \big| \leq \| f \|_{\infty} (c-c^{(1+k_0^{-1})})^{-1}
$$
for all $x \in C^c$. Taking the supremum over $x$, the claim follows.
\end{proof}

By construction, $S_0$ contains precisely one point of each 
$P_0$-essential class, and thus there are no asymptotic dynamical traps with 
respect to $S_0$. By Proposition \ref{condition nr} and \eqref{p eps numeric} we can thus
compute $\hat p_\eps(x_i,x_j)$ in a numerically stable way.  In fact, 
the perturbative nature of the problem makes the following 
Newton scheme particularly useful. 

Let 
$\bar P_\eps = (p_\eps(x,y))_{x,y \in S_0}$ denote the restriction of 
$P_\eps$ to $S_0^c$, and set $A_\eps = 1- \bar P_\eps$. 
By \eqref{matrix formulation}, we need to find $A_\eps^{-1}$. 
We use $B_0 = A_0^{-1}$ as a seed for the Newton 
iteration, and employ the usual recursion 
$B_{k+1} = 2 B_k - B_k A_\eps B_k$. 
By putting $\tilde B_k = B_k A_0$, we find $\tilde B_0 = 1$ and 
$\tilde B_{k+1} = 2 \tilde B_k - \tilde B_k A_0^{-1} A_\eps \tilde B_k$. So, 
$\tilde B_k$ is a polynomial in $A_0^{-1} A_\eps$, and we can use the resulting 
commutativity to obtain 
\be \label{newton 2}
B_{k+1} - A_\eps^{-1} = - A_0^{-1} A_\eps (B_k - A_\eps^{-1}) A_0 (B_k - A_\eps^{-1})
\ee
for all $k$. In the special case $k=0$, this can be transformed to 
\be \label{newton 1}
B_1 - A_\eps^{-1} = A_0^{-1} (A_\eps - A_0) (A_\eps^{-1} - A_0^{-1}) = 
A_0^{-1} (\bar P_0 - \bar P_\eps) (A_\eps^{-1} - A_0^{-1}).
\ee
Thus 
$$
\|B_{k+1} - A_\eps^{-1} \| \leq 2 \kappa(A_0) \| \|B_{k} - A_\eps^{-1} \|^2
$$
and 
$$
\| B_1 - A_\eps^{-1} \| \leq \|A_0^{-1}\| (\|A_0^{-1}\| + \|A_\eps^{-1}\|)  \| \bar P_\eps - \bar P_0 \|.
$$
Proposition \ref{condition nr} guarantees that we can choose $\eps$ sufficiently small so that $B_k$ converges to $A_\eps^{-1}$ very quickly. 
To illustrate this, we restrict ourselves to the special case where 
$P_\eps = P_0 + \eps R_\eps$ with the matrix $R_\eps$ bounded uniformly in $\eps > 0$.
Then, $\|B_1 - A_\eps^{-1}\| \leq c \eps$ for some $c > 0$, and 
$$
\|B_{k+1} - A_\eps^{-1}\| \leq (2 \kappa(A_0))^{2^{k-1}+1} (c \eps)^{2^{k}}.
$$
This means that when we are interested only in transitions 
of size $\eps^n$ or bigger, we only have to calculate logarithmically (in $n$) many 
$B_k$. Therefore, it might seem that all is well, but this is not entirely so.  

The reason is that a subtle problem arises from the multi-scale structure of the 
dynamics: at a given metastable time scale, it is in general not obvious 
what computational accuracy we need to achieve in order to obtain 
the asymptotically correct dynamics on longer metastable time scales. This phenomenon
can best be explained by an example. 
\bfig
\begin{center}
\includegraphics[width= 0.5 \textwidth]{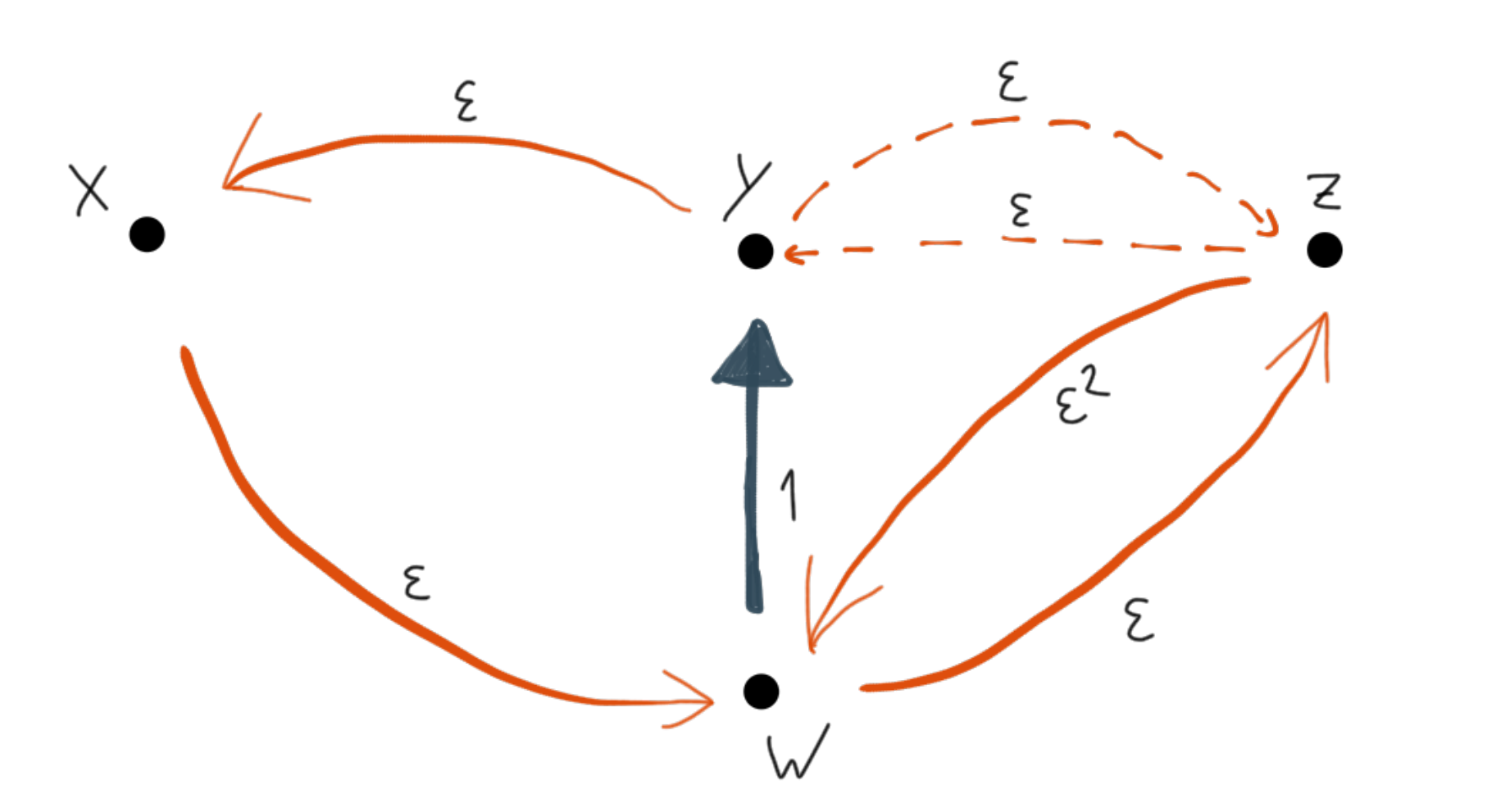} 
\end{center}
\caption{Schematic drawing of a perturbed Markov chain. Leading order transition
probabilities are written on the arrows. With or without the dashed arrow, we 
have $\hat p_\eps(x,z) = \eps^s$ and $\hat p_\eps(x,y) = \eps$, so transitions 
from 
}
\label{figure example}
\efig

Figure \ref{figure example} shows a graphical representation of a couple of 
metastable Markov chains. For both of them, $S = \{x,y,z,w\}$, and both of them have transition probabilities corresponding to the solid arrows: 
$p_\eps(x,w) = p_\eps(w,z) = p_\eps(y,x) = \eps$, 
$p_\eps(w,y)=1-\eps$, and 
$p_\eps(z,w) = \eps^2$. Only one of them has the
dashed arrows, i.e.\ $p(z,y) = p(y,z) = \eps$. All other transition probabilities
are zero except those mapping a point to itself, which are adjusted to give a 
stochastic matrix. 
With or without the dashed arrows, $\{x\}$, 
$\{y\}$ and $\{z\}$ are the $P_0$-essential classes, 
while $w$ is $P_0$-transient. Also in both cases, 
$\hat p_\eps(x,z) = \eps^2$, while $\hat p_\eps(x,y) = \eps$. 
So on the first metastable time scale, transitions from $x$ to $z$ play no role. 
But whether or not we can stop our computation of $\hat p_\eps(x,z)$ 
after reaching order $\eps$ depends on the presence of the dashed 
arrows.

If the dashed arrows are present, we can stop the computation of $\hat p_\eps$ after 
reaching order $\eps$: on the next (and final) metastable time scale, we will have 
$\check p_\eps(x,y) = \check p_\eps(z,y) = 1$ and 
$\check p_\eps(y,x) = \check p_\eps(y,z) = 1/2$. 
$z$ will be connected to $x$ via $y$, by transition probabilities of order one. 

However, if the dashed arrows are absent, stopping the 
calculation at order $\eps$ leads to an effective Markov chain where 
$z$ cannot be reached from $x$, and thus to wrong results on the next metastable 
time scale. In the correct dynamics on that time scale $x$ and $y$ form a new 
effective metastable state, and transitions between it 
and $z$ are (after rescaling) of order $\eps$. For this to be resolved correctly, 
the transition from $x$ to $z$ of order $\eps^2$ needs to be present already in the 
effective dynamics on the first metastable time scale. 

In the simple example at hand it is easy to directly figure out what is going on, but 
to decide when a given approximation of $\hat p_\eps$ is good enough 
to give correct dynamical results on all further metastable time scales 
for general chains on large state spaces is a subtle problem. 
Here we only give a necessary condition, about which we conjecture that it is also sufficient, and which is accessible 
to numerical validation.
Let us write $\hat \bbP_{\eps,a}$ for path measure of a given approximation to 
the chain $\hat X^{(\eps)}$. By Theorem \ref{escape probs}, 
$\hat q_\eps(E_i,E_j) \simeq \hat \bbP_\eps^{x_i} (\tau_{x_j} < \tau_{x_i})$ when $x_i$ is 
the representative from $E_i$ and $x_j$ the representative from $E_j$, and thus 
$\hat \mu(x_i) \simeq \mu(E_i)$ for all $i$. So in order to obtain the correct
asymptotic stationary distribution for our approximate chain, 
we have to increase the accuracy at least until 
\be \label{necessary cond}
\hat q_\eps(E_i,E_j) \simeq \hat \bbP_{\eps,a}^{x_i} (\tau_{x_j} < \tau_{x_i}).
\ee 
It would not be surprising if this were already sufficient for some sort of agreement of the 
metastable dynamics on all further metastable time scales. Since in general the 
escape probabilities do not characterize the transition probabilities of a Markov chain, 
a proof of this conjecture is not immediate, and we do not pursue this any further here. 
Instead, we discuss how to check \eqref{necessary cond} numerically. 

By \eqref{q eps numeric}, the numerically tricky part in computing $\hat q_\eps(E_i,E_j)$ is 
$\bbP_\eps^x(\tau_{E'}^+ < \tau_x^+)$.
Since $\{x\} \cup E'$ will not intersect all $P_0$-essential classes unless there are only
two of them, we cannot use Proposition \ref{condition nr} this time, and indeed in most 
situations a direct calculation of \eqref{matrix formulation} will be numerically 
unreliable.
However, for the 
very same reason, namely since $C$ intersects only two $P_0$-essential classes, 
we can successively lift these traps and 
arrive at a simplified chain without traps for which the probability 
of hitting $E'$ before $x$ is asymptotically equivalent to the original one. 

The basic step in this procedure is the following. 
Assume that $E$
is a $P_0$-essential class of a perturbed Markov chain $X^{(\eps)}$, and that $E \neq S$. 
We define a new 
Markov chain $\tilde X^{(\eps)}$ on the state space $\tilde S = (S \setminus E) \cup \{E\}$ 
by its transition probabilities $\tilde p_\eps(x,y)$, where 
$\tilde p_\eps(x,y) = p_\eps(x,y)$ whenever $x,y \in S \setminus E$, and 
\be \label{hat transitions}
\tilde p_\eps (x,E) := \sum_{z \in E} p_\eps(x,z), \qquad 
\tilde p_\eps(E,x) :=  
  \frac{1}{Z_\eps(E)} \sum_{z \in E} \nu_E(z) p_\eps(z,x), \qquad \tilde p(E,E):=0
\ee
for all $x \in S \setminus E$. Here, 
$Z_\eps(E) = \sum_{z \in E, y \notin E} \nu_E(z) p_\eps(z,y)$ is the 
normalization that ensures that $\tilde P$ is a stochastic matrix. We say that the 
traps in $E$ (with respect to $\cup(\caE\setminus\{E\})$) have been lifted in $\tilde X^{(\eps)}$. 
This terminology is justified by 

\begin{theorem} \label{lifting the trap}
Let $X^{(\eps)}$ be a perturbed Markov chain, $E$ a $P_0$-essential class 
of $X^{(\eps)}$, and $\tilde X^{(\eps)}$ the Markov chain where $E$ has been lifted. \\
a) Let $A,B \subset S \setminus E$. 
Then for all $z \in S \setminus E$, $\bbP^z_\eps (\tau_B < \tau_A) \simeq \tilde \bbP^z_\eps (\tau_B < \tau_A)$, while for $z \in E$, 
$\bbP^z_\eps (\tau_B < \tau_A) \simeq \tilde \bbP^E_\eps (\tau_B < \tau_A)$.
\\
b) If $X^{(\eps)}$ is regular, then  
$\tilde X^{(\eps)}$ is a regular perturbed Markov chain. \\
c)  If $X^{(\eps)}$ is regular, then 
either $E$ is a $\tilde P_0$-transient state, or $E$ is an element of a 
$\tilde P_0$-essential class that 
contains at least one further element $z \in F$. In the latter case, the number 
of $\tilde P_0$-transient states is strictly smaller than the number of 
$P_0$-transient states. 
\end{theorem}

\begin{proof}
Consider the chain $Y^{(\eps)}$ with state space $S$ and 
transition matrix $R_\eps$, where $r_\eps(x,y) = p_\eps(x,y)$ when $x \notin E$ and 
$r_\eps(x,y) = \bbP^x(\tau_{E^c}=y)$ when $x \in E$. Denoting its path measure 
by $\bbP_{Y,\eps}$, Proposition \ref{effective chain} gives
$\bbP_{Y,\eps}^z(\tau_B < \tau_A) = \tilde \bbP_\eps^z(\tau_B < \tau_A)$ 
for all $z \in S$. We now define $\tilde Y$ by 
replacing $r_\eps(x,y)$ with 
$\frac{1}{Z_\eps(E)} \sum_{x \in E} \nu_E(x) p_\eps(x,y)$ for $x \in E$, 
and keeping them the same if $x \notin E$.  
Then \eqref{magic formula 2} implies that 
$r_\eps(x,y) \simeq \tilde r_\eps(x,y)$ for all $x,y \in S$, and thus 
Theorem \ref{chain comparison} gives 
$\bbP_{\tilde Y,\eps}^z(\tau_B < \tau_A) 
\simeq \bbP_{Y,\eps}^z(\tau_B < \tau_A)$. 
Finally, noting that $\tilde r_\eps(z,w)$ does not depend on $z$ 
whenever $z \in E$, we can replace all $z \in E$ by a single 
state $\{E\}$, and claim a) follows.  

For b), note that by regularity of the chain, $Z_\eps(E) \sim \sum_{z \in E} \nu_E(z) p_\eps(z,x)$ for all $x \notin E$. So the quotient 
in \eqref{hat transitions} either converges or diverges to infinity as $\eps 
\to 0$. Since it is bounded by $1$ by construction, the latter is not an option,
and the $\tilde p_\eps$ converge. So the Markov chain defined by them is a 
perturbed Markov chain. Finally, this Markov chain is again regular, since 
products of its elements can be written as weighted 
sums of products of the $p_\eps$ with nonnegative weights. We have shown  b). 

For c), note that by b) $\lim_{\eps \to 0} \tilde P_\eps$ exists, and since 
$\sum_{y \notin E} \tilde p_\eps(E,y) = 1$, there must be at least one state
$y \in S \setminus E$ with $\lim_{\eps \to 0} \tilde p_\eps(E,y) > 0$. 
Lemma \ref{direct path lemma} implies that if $E'$ is a $P_0$-essential class 
with $E \neq E'$, then all direct paths from $E'$ to $E$ are $P_0$-irrelevant. So 
if one of the elements $y$ with  $\lim_{\eps \to 0} \tilde p_\eps(E,y) > 0$ 
is connected to a different $P_0$-essential class via a $P_0$-relevant direct path, 
then $E$ is 
$\tilde P_0$-transient. On the other hand, if no $y$ is connected to any 
$E' \neq E$ by such a direct path, then each such $y$ must be an element of $F$, 
and must be connected to $E$ by a $P_0$-relevant  
direct path. It follows that 
$y$ is in the same $\tilde P_0$-essential class as $E$, 
and thus not a $\tilde P_0$-transient state. The claim follows.
\end{proof}

Using Theorem \ref{lifting the trap}, we can now give a general recursive 
algorithm for numerically computing expressions $h_{B,A}(z)$ of the form given in 
\eqref{committor def} simultaneously for all $z \in S$, up to asymptotic equivalence:

\begin{enumerate}
\item  Determine the set $\caE_0$ of all $P_0$-essential classes 
not intersecting $A \cup B$. 
\item If $\caE_0 = \emptyset$, compute 
$h_{A,B}$ by solving the well-conditioned 
linear system \eqref{possibly ill conditioned}. Finish the algorithm.
\item Compute the $P_0$-stationary measures
$\nu_E$ for each $E \in \caE_0$. 
\item Lift all the traps in $E \in \caE_0$ by 
\eqref{hat transitions}. 
This results in a new state space, where all
elements of $E$ are replaced by a single state $E$. Keep track of 
the elements of the original state space that become lumped into 
$E$.  
\item  Return to (1) with the new state space.   
\end{enumerate}

We note that steps (3) and (4) are trivial to parallelize. 
By Theorem \ref{lifting the trap} c), each step either decreases the number of 
$P_0$-essential classes in the chain, or leaves it unchanged and decreases the number of 
transient states. We thus see that the algorithm terminates. Once it does (in step 2), 
we know $h_{A,B}(\tilde z)$ for all $\tilde z$ in the final state space $\tilde S$. Theorem 
\ref{lifting the trap} a) now guarantees that $h_{A,B}(z) \simeq h_{A,B}(\tilde z)$ for 
all states $z$ of the chain from the previous step that were collapsed into $\tilde z$. 
Thus we can recursively go backwards until we reach the original state space, where we 
now know all $h_{A,B}(z)$ up to asymptotic equivalence. 
In particular, this gives  a stable algorithm for the 
asymptotic numerical approximation of the coefficients 
$\hat q_\eps$. Since the expressions $\hat \bbP_{\eps,a}^{x_i} (\tau_{x_j} < \tau_{x_i})$
are also escape probabilities (for a different Markov chain), 
we can compute them by the same algorithm. If they agree with $\hat q_\eps(E_i,E_j)$ to 
leading order in $\eps$, our necessary criterion is met and the approximate chain has the 
same asymptotic stationary measure as the true one. 

Another useful aspect of our algorithm is that the $\hat q_\eps$ 
determine the limiting stationary 
distribution of the chain through the formula 
$$
\frac{1}{\mu_\eps(E)} \simeq \sum_{E' \in \caE} \frac{\hat q_\eps(E,E')}{\hat q_\eps(E',E)},
$$
which is derived in analogy to \eqref{generic u expression}, using Proposition 
\ref{asymptotic detailed balance}. Computing 
the stationary distribution of a large Markov chain with many metastable sets 
is a very important problem in practice. For example, it is how internet 
search engines compute page importance ranks. As a consequence, there 
has been tremendous activity in the computer science community on the topic. 
Most of the developments seem to be based on a seminal paper by 
Simon and Ando \cite{AnSi}. 
Seemingly independently, the problem has been treated by a much smaller group of 
people in mathematical economy, starting with \cite{Young} and with significant recent
progress by Wicks and Greenwald \cite{WiGr1,WiGr2}. 

Both approaches are based on formula \eqref{quot}, which itself is closely 
related to \eqref{possibly ill conditioned}. 
In the literature 
following \cite{AnSi} and \cite{Mey}, this leads to what is 
known as the method of the stochastic complement. For a finite Markov chain $X$ on a 
state space $S$, the first step of the method is to 
decompose $S$ into disjoint sets $S_1, \ldots, S_n$. 
Equation \eqref{quot} with $A = S_j$ then allows to compute  
\be \label{Andos}
\hat p(x,y) := \bbP^x(X_{\tau_{S_j}^+}=y)
\ee
for $x,y \in S$ by using matrix multiplications and by computing the 
inverse of the matrix $(1 - P|_{S_j^c})$.
The $\hat p(x,y)$ are the transition probabilities of an effective Markov  
chain only running inside $S_j$. Writing $\nu_j$ for the stationary distribution of the 
effective chain, and $\mu$ for the full stationary distribution, it can be shown that
\be \label{ando1}
\mu(x) = \xi_j \nu_j(x)
\ee
for all $x \in S_j$, where $(\xi_j)_{j \leq n}$ is 
the stationary distribution of the Markov chain with state space $\{S_1, \ldots, S_n\}$ 
and transition probabilities 
\be \label{ando2}
q(S_i,S_j) = \sum_{x \in S_i, y \in S_j} \nu_i(x) p(x,y).
\ee
Equation \eqref{ando1} is similar to the statements of our Corollary \ref{structure}, with 
the $\xi_j$ taking the role of $\mu_\eps(E)$, and the $\nu_j(x)$ the role of $\nu_E(x)$. 
Equation \eqref{ando2} is in analogy to the expression 
\be \label{fo1}
\hat p_\eps(x_i,x_j) \simeq \sum_{w \in E_i, z \notin E_i} \nu_{E_i}(w) p_\eps(w,z) 
\bbP_\eps^z(X_{\tau_{S_0}} = x_j)
\ee
that we get for $\hat p_\eps(x_i, x_j)$ when
combining \eqref{p hat} and \eqref{explanation}. The drawback of the 
method is that a priori, we have no control over the numerical difficulty of 
computing $(1 - P|_{S_j^c})^{-1}$. For example, 
let $S_1$ consist of two elements $x,y$. 
Then $\hat p(x,y) = \bbP^x(\tau_y^+ < \tau_x^+)$, 
and thus the computation of $\hat p(x,y)$ is no easier than the problem we have 
treated in the present paper; in particular, if $X$ is a perturbed Markov chain and 
$x$ and $y$ are in different $P_0$-essential classes, the matrix $(1 - P|_{S_j^c})$ will
become singular as $\eps \to 0$. Therefore without any further assumptions, 
the theory of Simon and Ando as it stands gives no numerically feasible way of computing $\mu$.  

A suitable such further assumption is to choose the decomposition in a way that makes 
all transitions between different $S_j$ small. The situation where this is possible 
has been treated already in \cite{AnSi}, and is nowadays known as a the theory of 
nearly reducible (or nearly decomposable) Markov chains. 
In the framework of the present paper, a perturbed Markov chain is nearly reducible
if for each $y \in S$ there exists a unique $P_0$-essential class $E(y)$ so that 
all $P_0$-relevant paths from $y$ to $S \setminus F$ end in $E(y)$. In the terminology
of \cite{BEGK1}, this means that the local valleys corresponding to the maximal 
metastable set $S_0 = \{x_1, \ldots, x_n\}$ from Section \ref{dynamics} do not intersect. 
When a Markov chain is nearly reducible, it is known (and follows from 
\eqref{unperturbed 2} in our case) that we can ignore transitions between different $S_j$
for the approximate computation of the $\nu_j$; in the case of perturbed Markov chains and 
when each $S_j$ contains exactly one $P_0$-essential class $E_j$, this means  
$\nu_j \approx \nu_{E_j}$. The reduced chain  
\eqref{ando2} is then similar to our $\hat X_\eps$, and by a
recursive algorithm similar to the one given in the present section, the stationary measure 
$\mu$ can be computed. 

So in the context of nearly reducible Markov chains, the contribution of 
our work is on the one hand a systematic, rigorous asymptotic theory, and on the 
other hand an extension to the case where the Markov chain no longer needs to be nearly
reducible: in the latter case, the $E_j$ take the role of the $S_j$, and the presence of the 
transient set is accounted for by replacing \eqref{ando2} by \eqref{fo1}, together 
with a recipe to compute the escape probabilities contained in the latter equation. 

The second approach that we are aware of which uses \eqref{quot} is the recent work by 
Wicks and Greenwald \cite{WiGr1, WiGr2}, who call their approach the method of the 
stochastic quotient. 
They work in the situation where $P_\eps = P_0 + \eps R_\eps$ with 
bounded corrector matrix $R_\eps$, and they do not need to assume 
almost decomposability. As we do, they pick a representative $x$ 
from each $P_0$-essential class $E$. Then they apply \eqref{quot} 
with $A = \{x\} \cup S \setminus E$, i.e. they compute the probabilities to either leave
$E$ at a given $y \notin E$, or to return to $x$. The leading order of this quantity 
can be computed efficiently by a matrix calculation, since the matrices 
$(1 - P_\eps|_{A^c})^{-1}$ remain bounded as $\eps \to 0$ 
thanks to the absence of $x$ from $A^c$. Indeed, as Wicks and Greenwald note, 
it suffices to invert $(1 -P_0|_{A^c})$. This construction leads to an
effective chain where the class $E$ is replaced by a $P_0$-essential 
class containing just the one element $x$. They do this construction for all 
$P_0$-essential classes, and indeed also for transient communicating classes. 
After that, they rescale transition probabilities out of each of the 
(now trivial) $P_0$-essential classes much like we do in \eqref{hat transitions}, 
keeping track of the factors by which they speed up each individual trap. This results 
in a Markov chain with fewer $P_0$-essential states, or fewer transient states. 
Recursively iterating the 
procedure while always keeping track of the rescaling factors, they arrive at a stable
algorithm for computing the stationary distribution. 

It is obvious that the algorithm of Wicks and Greenwald and ours share quite similar ideas. 
The difference is that while our algorithm lifts metastable traps 
completely, the Wicks-Greenwald algorithm keeps one point in each trap. 
The advantage of the Wicks-Greenwald algorithm is that the whole stationary 
distribution can be computed at once, while in our algorithm one has to 
compute $\hat q(E,E')$ separately for each pair $E,E'$. The advantage of our approach is 
that it is local: if we are only interested in the relative importance of two 
given states $x \in E$ and $y \in E'$, we need only compute the ratio 
$\hat q(E,E')/\hat q(E',E)$. Depending on the structure of the chain, this can be done
by lifting only a tiny fraction of the traps present in the state space. An additional
advantage of our approach is of course that we also obtain information about the metastable
dynamics, information which is not contained in the stationary distribution alone.

\end{document}